\documentclass{arxiv}
\usepackage{amsmath}
\usepackage{amssymb}
\usepackage{amsfonts}
\usepackage{amstext}
\usepackage{amsopn}
\usepackage{amsxtra}
\usepackage{graphicx}
\usepackage[latin1]{inputenc}
\usepackage{dsfont}
\usepackage{mathrsfs}
\usepackage{multirow}
\newcommand\R{{\ensuremath {\mathbb R} }}
\newcommand\C{{\ensuremath {\mathbb C} }}

\newcommand\Z{{\ensuremath {\mathbb Z} }}

\newcommand\1{{\ensuremath {\mathds 1} }}
\renewcommand\phi{\varphi}
\newcommand{\alp}{\boldsymbol{\alpha}}
\newcommand{\gH}{\mathfrak{H}}

\newcommand{\gS}{\mathfrak{S}}
\newcommand{\wto}{\rightharpoonup}
\renewcommand{\to}{\rightarrow}
\newcommand{\cP}{\mathcal{P}}

\newcommand{\cD}{\mathcal{D}}

\newcommand{\cL}{\mathcal{L}}

\newcommand{\tr}{{\rm tr}\,}

\newcommand{\cC}{\mathcal{C}}

\newcommand{\Spu}{{\rm Spu}}
\newcommand{\Conv}{{\rm Conv}}
\newcommand{\curlL}{\mathscr{L}}
\newcommand{\curlB}{\mathscr{B}}
\newcommand{\curlC}{\mathscr{C}}
\newcommand{\LKB}{L_{KB}}
\newcommand{\LAB}{L_{AB}}

\newcommand{\LDKB}{L_{DKB}}
\newcommand\ii{{\ensuremath {\infty}}}
\newcommand\pscal[1]{{\ensuremath{\left\langle #1 \right\rangle}}}
\newcommand{\norm}[1]{ \left| \! \left| #1 \right| \! \right| }

\renewcommand{\tr}{{\rm Tr} }

\begin{document}
\date{December 11, 2008}
\title{\Large Spectral Pollution and How to Avoid It\\ \small (With Applications to Dirac and Periodic Schrödinger Operators)}
\authormark{Mathieu LEWIN \& \'Eric S\'ER\'E}
\runningtitle{Spectral Pollution and How to Avoid It}

\author{Mathieu LEWIN}
\address{CNRS and Laboratoire de Mathématiques (CNRS UMR 8088), Universit{\'e} de Cergy-Pontoise, 2, avenue Adolphe Chauvin, 95 302 Cergy-Pontoise Cedex - France.\\ Email: \email{Mathieu.Lewin@math.cnrs.fr}}

\author{\'Eric S\'ER\'E}
\address{Ceremade (CNRS UMR 7534), Universit{\'e} Paris-Dauphine, Place
du Mar{\'e}chal de Lattre de Tassigny, 75775 Paris Cedex 16 - France.\\ Email: \email{sere@ceremade.dauphine.fr}}

\maketitle

\bigskip

\begin{abstract}
This paper, devoted to the study of spectral pollution, contains both abstract results and applications to some self-adjoint operators with a gap in their essential spectrum occuring in Quantum Mechanics. 

First we consider Galerkin basis which respect the decomposition of the ambient Hilbert space into a direct sum $\gH=P\gH\oplus(1-P)\gH$, given by a fixed orthogonal projector $P$, and we localize the polluted spectrum exactly. This is followed by applications to periodic Schrödinger operators (pollution is absent in a Wannier-type basis), and to Dirac operator (several natural decompositions are considered).

In the second part, we add the constraint that within the Galerkin basis there is a certain relation between vectors in $P\gH$ and vectors in $(1-P)\gH$. Abstract results are proved and applied to several practical methods like the famous \emph{kinetic balance} of relativistic Quantum Mechanics.

\smallskip

\noindent{\scriptsize\copyright~2008 by the authors. This paper may be reproduced, in its entirety, for non-commercial~purposes.}
\end{abstract}

\tableofcontents

\section*{Introduction}\addcontentsline{toc}{section}{Introduction}
This paper is devoted to the study of \emph{spectral pollution}. This phenomenon of high interest occurs when one approximates the spectrum of a (bounded or unbounded) self-adjoint operator $A$ on an infinite-dimensional Hilbert space $\gH$, using a sequence of finite-dimensional spaces. Consider for instance a sequence $\{V_n\}$ of subspaces of the domain $D(A)$ of $A$ such that $V_n\subset V_{n+1}$ and $P_{V_n}\to1$ strongly (we denote by $P_{V_n}$ the orthogonal projector on $V_n$). Define the $n\times n$ matrices $A_n:=P_{V_n}AP_{V_n}$. It is well-known that such a Galerkin method may in general lead to \emph{spurious eigenvalues}, i.e. numbers $\lambda\in\R$ which are limiting points of eigenvalues of $A_n$ but do not belong to $\sigma(A)$. This phenomenon is known to occur in gaps of the essential spectrum of $A$ only.

Spectral pollution is an important issue which arises in many different practical situations. It is encountered when approximating the spectrum of perturbations of periodic Schrödinger operators \cite{BouLev-07} or Strum-Liouville operators \cite{StoWei-93,StoWei-95,AceGheMar-06}. It is a very well reported difficulty in Quantum Chemistry and Physics in particular regarding relativistic computations \cite{DraGol-81,Grant-82,Kutzelnigg-84,StaHav-84,DyaFae-90,Pestka-03,Shaetal-04}. It also appears in elasticity, electromagnetism and hydrodynamics; see, e.g. the references in \cite{Bou-07}. Eventually, it has raised as well a huge interest in the mathematical community, see, e.g., \cite{LevSha-04,DavPlu-04,BouLev-07,Hansen-08,Descloux-81,Pokrzywa-79,Pokrzywa-81}.

In this article we will study spectral pollution from a rather new perspective. Although many works focus on how to determine if an approximate eigenvalue is spurious or not (see, e.g., the rather successful second-order projection method \cite{LevSha-04,BouLev-07}), we will on the contrary concentrate on finding conditions on the sequence $\{V_n\}$ which ensure that there will not be any pollution at all, in a given interval of the real line.

Our work contains two rather different aspects. On the one hand we will establish some theoretical results for abstract self-adjoint operators: we characterize exactly (or partially) the polluted spectrum under some specific assumptions on the approximation scheme as will be explained below. On the other hand we apply these results to two important cases of Quantum Physics: perturbations of periodic Schrödinger operators and Dirac operators. 
For Dirac operators, we will show in particular that some very well-known methods used by Chemists or Physicists  indeed allow to partially avoid spurious eigenvalues in certain situations, or at the contrary that they are theoretically of no effect in other cases.

\bigskip

Let us now summarize our results with some more details. 

Our approach consists in adding some assumptions on the approximating scheme. We start by considering in Section \ref{sec:splitting} a fixed orthogonal projector $P$ acting on the ambiant Hilbert space $\gH$ and we define $P$-spurious eigenvalues $\lambda$ as limiting points obtained by a Galerkin-type procedure, in a basis which respects the decomposition associated with $P$. This means $\lambda=\lim_{n\to\ii}\lambda_n$ with  $\lambda\notin\sigma(A)$ and $\lambda_n\in\sigma(P_{V_n}AP_{V_n})$, where $V_n=V_n^+\oplus V_n^-$ for some $V_n^+\subset \gH^+:=P\gH$ and $V_n^-\subset \gH^-:=(1-P)\gH$. We show that, contrarily to the general case and depending on $P$, there might exist an interval in $\R$ in which there is never any pollution occuring. More precisely, we exactly determine the location of the polluted spectrum in Section \ref{sec:splitting_general} and we use this in Section \ref{sec:splitting_condition} to derive a simple criterion on $P$, allowing to completely avoid the appearence of spurious eigenvalues in a gap of the essential spectrum of $A$.

Then we apply our general result to several practical situations in Section \ref{sec:applications_splitting}. We in particular show that the usual decomposition into upper and lower spinors \emph{a priori} always leads to pollution for Dirac operators. We also study another  decomposition of the ambient Hilbert space which was proposed by Shabaev et al \cite{Shaetal-04} and we prove that the set which is free from spectral pollution is larger than the one obtained from the simple decomposition into upper and lower spinors. Eventually, we prove that choosing the decomposition given by the spectral projectors of the free Dirac operator is completely free of pollution. For the convenience of the reader, we have summarized all these results in Table \ref{table:summary_splitting}. 

As another application we consider in Section \ref{sec:Periodic} the case of a periodic Schrödinger operator which is perturbed by a potential which vanishes at infinity. We prove again that choosing a decomposition associated with the unperturbed (periodic) Hamiltonian allows to avoid spectral pollution, as was already demonstrated numerically in \cite{CanDelLew-08b} using Wannier functions.

\begin{table}[t]
\centering
\begin{tabular}{|c|c|c|}
\hline
\textbf{Imposed splitting} & \multirow{2}{*}{\textbf{External potential $V$}} & \textbf{Spurious spectrum}\\
 \textbf{of Hilbert space} &  & \textbf{in the gap $(-1,1)$}\\
\hline\hline
\textit{none} & any & $(-1,1)$\\
\hline
\textit{upper/lower spinors} & $V=0$ & $\emptyset$\\
\cline{2-3}
\multirow{3}{*}{$\left(\begin{array}{c}
\phi_n\\ 0\end{array}\right),\ \left(\begin{array}{c}
0\\ \chi_n\end{array}\right)$} & \multirow{2}{*}{$V$ bounded} & $(-1,-1+\sup(V)]\qquad\qquad$\\
 & & $\qquad\qquad\cup[1+\inf(V),1)$\\
\cline{2-3}
 & unbounded (ex: Coulomb) & $(-1,1)$\\
\hline
\textit{dual decomposition \cite{Shaetal-04}} & $V=0$ & $\emptyset$\\
\cline{2-3}
\multirow{2}{*}{$\left(\begin{array}{c}
\phi_n\\ \epsilon\sigma\cdot p\,\phi_n\end{array}\right),\; \left(\begin{array}{c}
-\epsilon\sigma\cdot p\,\chi_n\\ \chi_n\end{array}\right)$} & \multirow{2}{*}{$V$ bounded} & $\displaystyle(-1,-{2}/{\epsilon}+1+\sup(V)]\qquad$\\
 & & $\displaystyle\qquad\cup[{2}/{\epsilon}-1+\inf(V),1)$\\
\cline{2-3}
$0<\epsilon\leq1$ & unbounded (ex: Coulomb) & $(-1,1)$\\
\hline
\textit{free decomposition} & \multirow{2}{*}{any} & \multirow{2}{*}{$\emptyset$}\\
$P^0_+\Psi_n,\ P^0_-\Psi'_n$ & & \\
\hline
\end{tabular}
\caption{Summary of our results from Section \ref{sec:applications_splitting} for the Dirac operator $D^0+V$, when a splitting is imposed on the Hilbert space $L^2(\R^3,\C^4)$.}
\label{table:summary_splitting}
\end{table}

In Section \ref{sec:balanced}, we come back to the theory of a general operator $A$ and we study another method inspired by the ones used in quantum Physics and Chemistry. Namely, additionaly to a splitting as explained before, we add the requirement that there is a specific relation (named \emph{balance condition}) between the vectors of $\gH^-$ and that of $\gH^+$. This amounts to choosing a fixed operator $L:\gH^+\to\gH^-$ and taking as approximation spaces $V_n=V_n^+\oplus LV_n^+$. We do not completely characterize theoretically the possible spurious eigenvalues for this kind of methods but we give necessary and sufficient conditions which are enough to fully understand the case of the Dirac operator. In Quantum Chemistry and Physics the main method is the so-called \emph{kinetic balance} which consists in choosing $L=\sigma(-i\nabla)$ and the decomposition into upper and lower spinors. We show in Section \ref{sec:kinetic_balance} that this method allows to avoid spectral pollution in the upper part of the spectrum only for bounded potentials and that it does not help for unbounded functions like the Coulomb potential. We prove in Section \ref{sec:atomic_balance} that the so-called (more complicated) \emph{atomic balance} indeed allows to solve this problem also for Coulomb potentials, as was already suspected in the literature. Eventually, we show that the \emph{dual kinetic balance} method of \cite{Shaetal-04} is not better than the one which is obtained by imposing a splitting without \emph{a priori} adding a balance condition. Our results for balanced methods for Dirac operators are summarized in Table \ref{table:summary_balance}.

\begin{table}[t]
\centering
\begin{tabular}{|c|c|c|}
\hline
\textbf{Balance} & \multirow{2}{*}{\textbf{External potential $V$}} & \textbf{Spurious spectrum}\\
 \textbf{condition} &  & \textbf{in the gap $(-1,1)$}\\
\hline\hline
\textit{kinetic balance} & $V$ bounded with & \multirow{2}{*}{$(-1,-1+\sup(V)]$}\\
\multirow{3}{*}{$\left(\begin{array}{c}
\phi_n\\ 0\end{array}\right),\ \left(\begin{array}{c}
0\\ \sigma\cdot p\,\phi_n\end{array}\right)$} & $-1+\sup(V)<1+\inf(V)$ & \\
\cline{2-3}
 & $V(x)=-\frac{\kappa}{|x|}$, & \multirow{2}{*}{$(-1,1)$}\\
&  ${0<\kappa<\sqrt{3}/2}$ & \\
\hline
\textit{atomic balance} & $V$ such that & \multirow{4}{*}{$(-1,-1+\sup(V)]$}\\
\multirow{3}{*}{$\left(\begin{array}{c}
\phi_n\\ 0\end{array}\right),\ \left(\begin{array}{c}
0\\ \frac{1}{2-V}\sigma\cdot p\,\phi_n\end{array}\right)$} &$-\frac{\kappa}{|x|}\leq V(x)$ where & \\
 & $0\leq\kappa<\sqrt{3}/2$, & \\
&  and $\sup(V)<2$ & \\
\hline
\textit{dual kinetic balance} \cite{Shaetal-04} & \multirow{2}{*}{$V$ bounded} &  $\displaystyle(-1,-{2}/{\epsilon}+1+\sup(V)]\qquad$\\
\multirow{2}{*}{$\left(\begin{array}{c}
\phi_n\\ \epsilon\sigma\cdot p\,\phi_n\end{array}\right),\; \left(\begin{array}{c}
-\epsilon\sigma\cdot p\,\phi_n\\ \phi_n\end{array}\right)$} & &$\displaystyle\qquad\cup[{2}/{\epsilon}-1+\inf(V),1)$\\
\cline{2-3}
 & unbounded (ex: Coulomb) & $(-1,1)$\\
\hline
\end{tabular}
\caption{Summary of our results for the Dirac operator $D^0+V$ when a balance is imposed between vectors of the basis.}
\label{table:summary_balance}
\end{table}

We have tried to make our results sufficiently general that they could be applied to other situations in which there is a natural way (in the numerical sense) to split the ambiant Hilbert space in a direct sum $\gH=\gH^+\oplus\gH^-$. We hope that our results will provide some new insight on the spectral pollution issue.

\bigskip\noindent{\it Acknowledgements.} {\small The authors would like to thank Lyonell Boulton and Nabile Boussaid for interesting discussions and comments. The authors have been supported by the ANR project \emph{ACCQuaRel} of the french ministry of research.}

\section{Spectral pollution}
In this first section, we recall the definition of \emph{spectral pollution} and give some properties which will be used in the rest of the paper. Most of the material of this section is rather well-known \cite{Descloux-81,Shaetal-04,LevSha-04,DavPlu-04}.

In the whole paper we consider a self-adjoint operator $A$ acting on a separable Hilbert space $\gH$, with dense domain $D(A)$. 

\paragraph{Notation.}
For any finite-dimensional subspace $V\subset D(A)$, we denote by $P_V$ the orthogonal projector onto $V$ and by $A_{|V}$ the self-adjoint operator $V\to V$ which is just the restriction to $V$ of $P_VAP_V$. 

As $A$ is by assumption a self-adjoint operator, it is closed, i.e. the graph $G(A)\subset D(A)\times\gH$ is closed. This induces a norm $\norm{\cdot}_{D(A)}$ on $D(A)$ for which $D(A)$ is closed. For any $K\subset D(A)$, we will use the notation $\overline{K}^{D(A)}$ to denote the closure of $K$ for the norm associated with the graph of $A$, in $D(A)$. On the other hand we simply denote by $\overline{K}$ the closure for the norm of the ambient space $\gH$.

We use like in \cite{LevSha-04} the notation $\hat{\sigma}_{\rm ess}(A)$ to denote the essential spectrum of $A$ union $-\ii$ (and/or $+\ii$) if there exists a sequence of $\sigma(A)\ni\lambda_n\to-\ii$ (and/or $+\ii$). Finally, we denote by $\Conv(X)$ the convex hull of any set $X\subset\R$ and we use the convention that $[c,d]=\emptyset$ if $d<c$.

\begin{definition}[Spurious eigenvalues]\label{def:spurious}\it 
We say that $\lambda\in\R$ is a \emph{spurious eigenvalue} of the operator $A$ if there exists a sequence of finite dimensional spaces $\{V_n\}_{n\geq1}$ with $V_n\subset D(A)$ and $V_n\subset V_{n+1}$ for any $n$, such that 
\begin{enumerate}
\item[(i)] $\overline{\cup_{n\geq1}V_n}^{D(A)}=D(A)$;
\item[(ii)] $\displaystyle\lim_{n\to\ii}{\rm dist}\left(\lambda\,,\,\sigma(A_{|V_n})\right)=0$;
\item[(iii)] $\lambda\notin\sigma(A)$.
\end{enumerate}
We denote by $\Spu(A)$ the set of spurious eigenvalues of $A$.
\end{definition}

If needed, we shall say that \emph{$\lambda$ is a spurious eigenvalue of $A$ with respect to $\{V_n\}$} to further indicate a sequence $\{V_n\}$ for which the above properties hold true.
Note that (i) in Definition \ref{def:spurious} implies in particular that we have
$\overline{\cup_{n\geq1}V_n}=\gH$
since $D(A)$ is dense in $\gH$ by assumption.

\begin{remark}\label{rmk:domain}\it
As the matrix of $A$ in a finite-dimensional space only involves the quadratic form associated with $A$, it is possible to define spurious eigenvalues by assuming only that $V_n$ is contained in the form domain of $A$. Generalizing our results to quadratic forms formalism is certainly technical, although being actually useful in some cases (Finite Element Methods are usually expressed in this formalism). We shall only consider the simpler case for which $V_n\subset D(A)$ for convenience.
\end{remark}

\begin{remark}\label{rmk:compact-perturb}\it
If $\lambda$ is a spurious eigenvalue of $A$ with respect to $\{V_n\}$ and if $B-A$ is compact, then $\lambda$ is either a spurious eigenvalue of $B$ in $\{V_n\}$ or $\lambda\in\sigma_{\rm disc}(B)$. One may think that the same holds when $B-A$ is only $A$-compact, but this is actually \underline{not true}, as we shall illustrate below in Remark \ref{rmk:contre-exemple}.
\end{remark}

\begin{remark}\it
In this paper we concentrate our efforts on the spectral pollution issue, and we do not study how well the spectrum $\sigma(A)$ of $A$ is approximated by the discretized spectra $\sigma(A_{|V_n})$. Let us only mention that for every $\lambda\in\sigma(A)$, we have ${\rm dist}(\lambda,\sigma(A_{|V_n}))\to0$ as $n\to\ii$, provided that $\overline{\cup_{n\geq1}V_n}^{D(A)}=D(A)$ as required in Definition \ref{def:spurious}.
\end{remark}

The following lemma will be very useful in the sequel.
\begin{lemma}[Weyl sequences]\label{lem:Weyl} Assume that $\lambda$ is a spurious eigenvalue of $A$ in $\{V_n\}$ as above. Then there exists a sequence $\{x_n \}_{n\geq1}\subset D(A)$ with $x_n\in V_n$ for any $n\geq1$, such that
\begin{enumerate}
\item $P_{V_n}(A-\lambda)x_n\to 0$ strongly in $\gH$;
\item $\norm{x_n}=1$ for all $n\geq1$;
\item $x_n\wto0$ weakly in $\gH$.
\end{enumerate}
\end{lemma}
\begin{proof}
It is partly contained in \cite{DavPlu-04}.
Let $\lambda\in\Spu(A)$ and consider $x_n\in V_n\setminus\{0\}\subset D(A)$ such that $P_{V_n}(A-\lambda_n)x_n=0$ with $\lim_{n\to\ii}\lambda_n=\lambda$. Dividing by $\norm{x_n}$ if necessary, we may assume that $\norm{x_n}=1$ for all $n$ in which case $P_{V_n}(A-\lambda)x_n\to0$ strongly. As $\{x_n\}$ is bounded, extracting a subsequence if necessary we may assume that $x_n\wto x$ weakly in $\gH$. What remains to be proven is that $x=0$.
Let $y\in \cup_{m\geq1}V_m$.  Taking $n$ large enough we may assume that $y\in V_n$. Next we compute the following scalar product 
$$0=\lim_{n\to\ii}\pscal{P_{V_n}(A-\lambda)x_n,y}=\lim_{n\to\ii}\pscal{x_n,(A-\lambda)y}=\pscal{x,(A-\lambda)y}.$$
As $\cup_{m\geq1}V_m$ is dense in $D(A)$ for the norm of $G(A)$, we deduce that $\pscal{x,(A-\lambda)y}=0$ for all $y\in D(A)$.
Thus $x\in D(A^*)=D(A)$ and it satisfies $Ax=\lambda x$. Hence $x=0$ since $\lambda$ is not an eigenvalue of $A$ by assumption. 
\end{proof}

The next lemma will be useful to identify points in ${\rm Spu}(A)$.
\begin{lemma}\label{lem:reciproque_Weyl} 
Assume that $A$ is as above.
Let $(x_n^1,...,x_n^K)$ be an orthonormal system of $K$ vectors in $D(A)$ such that $x_n^j\wto0$ for all $j=1..K$. Denote by $W_n$ the space spanned by $x_n^1,..., x_n^K$. If $\lambda\in\R$ is such that 
$\lim_{n\to\ii}{\rm dist}\left(\lambda\,,\,\sigma(A_{|W_n})\right)=0,$
then
$\lambda\in{\rm Spu}(A)\cup\sigma(A).$
\end{lemma}
\begin{proof}
Consider any nondecreasing sequence $\{V_n\}$ such that $\overline{\cup_{n\geq1}V_n}^{D(A)}=D(A)$. Next we introduce $V'_1:=V_1$,  $m_1=0$ and we construct by induction a new sequence $\{V_n'\}$ and an increasing sequence $\{m_n\}$ as follows. Assume that $V'_n$ and $m_n$ are defined. As $x_m^k\wto0$ for all $k=1..j$, we have $\lim_{m\to\ii}\pscal{Ay,x_m^k}=0$ for all $y\in V'_n$ and all $k=1..K$. Hence the matrix of $A$ in $V'_n+W_m$ becomes diagonal by blocks as $m\to\ii$. Therefore there exists $m_{n+1}>m_n$ such that the matrix of $A$ in $V'_{n+1}:=V'_n+W_{m_{n+1}}$ has an eigenvalue which is at a distance $\leq1/n$ from $\lambda$.
As $V_n\subset V_n'$ for all $n$, we have $\overline{\cup_{n\geq1}V'_n}^{D(A)}=D(A)$. By construction we also have $\lim_{n\to\ii}{\rm dist}\left(\lambda\,,\,\sigma(A_{|V'_n})\right)=0$. Hence either $\lambda\in\sigma(A)$, or $\lambda\in {\rm Spu}(A)$.
\end{proof}

In the following we shall only be interested in the spurious eigenvalues of $A$ lying in the convex hull of $\hat\sigma_{\rm ess}(A)$. This is justified by the following simple result which tells us that pollution cannot occur below or above the essential spectrum.
\begin{lemma} Let $\lambda$ be a spurious eigenvalue of the self-adjoint operator $A$. Then one has
\begin{equation}
\tr\left(\chi_{(-\ii,\lambda]}(A)\right)=\tr\left(\chi_{[\lambda,\ii)}(A)\right)=+\ii.
\label{cond_pollution}
\end{equation}
Saying differently, $\lambda\in\Conv\left(\hat{\sigma}_{\rm ess}(A)\right)$.
\end{lemma}
\begin{proof}
Assume for instance $P:=\chi_{(-\ii,\lambda]}(A)$ is finite-rank. As $\lambda\notin\sigma(A)$, we must have $P=\chi_{(-\ii,\lambda+\epsilon]}(A)$ for some $\epsilon>0$.
Let $\{x_n\}$ be as in Lemma \ref{lem:Weyl}. As $P$ is finite rank, $Px_n\to0$ and $(A-\lambda)Px_n\to0$ strongly in $\gH$. Therefore $P_{V_n}(A-\lambda)P^\perp x_n\to0$ strongly. Note that $(A-\lambda)P^\perp\geq \epsilon P^\perp$, hence
$\pscal{P_{V_n}(A-\lambda)P^\perp x_n,x_n}=\pscal{P^\perp(A-\lambda)P^\perp x_n,x_n}\geq \epsilon\norm{P^\perp x_n}^2$. As the left hand side converges to zero, we infer $\norm{x_n}\to0$ which contradicts Lemma \ref{lem:Weyl}.
\end{proof}

We have seen that pollution can only occur in the convex hull of $\hat{\sigma}_{\rm ess}(A)$. Levitin and Shargorodsky have shown in \cite{LevSha-04} that \eqref{cond_pollution} is indeed necessary and sufficient.
\begin{theorem}[Pollution in all spectral gaps \cite{LevSha-04}]\label{thm:general}
Let $A$ be a self-adjoint operator on $\gH$ with dense domain $D(A)$. Then 
$$\overline{\Spu(A)}\cup\hat\sigma_{\rm ess}(A)=\Conv\left(\hat\sigma_{\rm ess}(A)\right).$$
\end{theorem}

\begin{figure}[h]
\centering
\includegraphics[width=9cm]{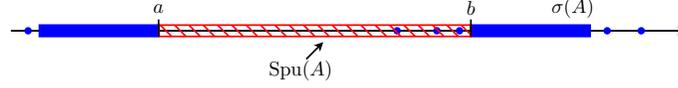}
\label{fig:pollution}
\caption{For an operator $A$ which has a spectral gap $[a,b]$ in its essential spectrum, pollution can occur in the whole gap.}
\end{figure}

\begin{remark}
As $J:=\Conv\left(\hat\sigma_{\rm ess}(A)\right)\setminus\hat\sigma_{\rm ess}(A)$ only contains discrete spectrum by assumption, Theorem \ref{thm:general} says that all points but a countable set in $J$ are potential spurious eigenvalues.
\end{remark}
\begin{remark}
It is easy to construct a sequence $V_n$ like in Definition \ref{def:spurious} such that ${\rm dist}\left(\lambda,\sigma(A_{|V_n})\right)\to0$ \emph{for all} $\lambda\in\Conv\left(\hat\sigma_{\rm ess}(A)\right)$, see \cite{LevSha-04}.
\end{remark}

Theorem \ref{thm:general} was proved for bounded self-adjoint operators in \cite{Pokrzywa-79} and generalized to bounded non self-adjoint operators in \cite{Descloux-81}. For the convenience of the reader, we give a short
\begin{proof}
Let $\lambda\in\Conv\left(\hat\sigma_{\rm ess}(A)\right)\setminus\hat\sigma_{\rm ess}(A)$ and fix some $a<\lambda$ and $b>\lambda$ such that $a,b\in\hat\sigma_{\rm ess}(A)$ (\emph{a priori} we might have $b=+\ii$ or $a=-\ii$). Let us consider two sequences $\{x_n\},\ \{y_n\}\subset D(A)$ such that $(A-a_n)x_n\to0$, $(A-b_n)y_n\to0$, $\norm{x_n}=\norm{y_n}=1$, $x_n\wto0$, $y_n\wto0$, $a_n\to a$ and $b_n\to b$. Extracting subsequences if necessary we may assume that $\pscal{x_n,y_n}\to0$ as $n\to\ii$. Next we consider the sequence $z_n(\theta):=\cos\theta\; x_n+\sin\theta\; y_n$ which satisfies $\norm{z_n(\theta)}\to1$ and $z_n(\theta)\wto0$ uniformly in $\theta$. We note that $\pscal{Az_n(0),z_n(0)}=a_n+o(1)$ and $\pscal{Az_n(\pi/2),z_n(\pi/2)}=b_n+o(1)$. Hence for $n$ large enough there exists a $\theta_n\in(0,\pi/2)$ such that $\pscal{Az_n(\theta_n),z_n(\theta_n)}=\lambda$.
The rest follows from Lemma \ref{lem:reciproque_Weyl}.
\end{proof}

\section{Pollution associated with a splitting of $\gH$}\label{sec:splitting}
As we have recalled in the previous section, the union of the essential spectrum and (the closure of) the polluted spectrum is always an interval: it is simply the convex hull of $\hat\sigma_{\rm ess}(A)$. It was also shown in \cite{LevSha-04} that it is possible to construct one sequence $\{V_n\}$ such that all possible points in $\Spu(A)$ are indeed $\{V_n\}$-spurious eigenvalues. 
But of course, not all $\{V_n\}$ will produce pollution. If for instance $P_{V_n}$ commutes with $A$ for all $n\geq1$, then pollution will not occur as is obviously seen from Lemma \ref{lem:Weyl}. The purpose of this section is to study spectral pollution if we add some assumptions on $\{V_n\}$. More precisely we will fix an orthogonal projector $P$ acting on $\gH$ and we will add the natural assumption that $P_{V_n}$ commute with $P$ for all $n$, i.e. that $V_n$ only contains vectors from $P\gH$ and $(1-P)\gH$. 

As we will see, under this new assumption the polluted spectrum (union $\hat\sigma_{\rm ess}(A)$) will in general be \emph{the union of two intervals}. Saying differently, by adding such an assumption on $\{V_n\}$, we can \emph{create a hole in the polluted spectrum}. A typical situation is when our operator $A$ has a gap in its essential spectrum. Then we will see that it is possible to give very simple conditions\footnote{Loosely speaking it must not be too far from the spectral projector associated with the part of the spectrum above the gap, as we will see below.} on $P$ which allow to completely avoid pollution in the gap. 

Note that our results of this section can easily be generalized to the case of a partition of unity $\{P_i\}_{i=1}^p$ of commuting projectors such that $1=\sum_{i=1}^pP_i$. Adding the assumption that $P_{V_n}$ commutes with all $P_i$'s, we would create $p$ holes in the polluted spectrum. This might be useful if one wants to avoid spectral pollution in several gaps at the same time.

\subsection{A general result}\label{sec:splitting_general}
We start by defining properly $P$-spurious eigenvalues.

\begin{definition}[Spurious eigenvalues associated with a splitting]\label{def_poll_P}\it
Consider an orthogonal projection $P:\gH\to\gH$. We say that $\lambda\in\R$ is a \emph{$P$-spurious eigenvalue} of the operator $A$ if there exist two sequences of finite dimensional spaces $\{V_n^+\}_{n\geq1}\subset P\gH\cap D(A)$ and $\{V_n^-\}_{n\geq1}\subset (1-P)\gH\cap D(A)$ with $V_n^\pm\subset V_{n+1}^\pm$ for any $n$, such that 
\begin{enumerate}
\item $\overline{\cup_{n\geq1}(V_n^-\oplus V_n^+)}^{D(A)}=D(A)$;
\item $\displaystyle\lim_{n\to\ii}{\rm dist}\left(\lambda,\sigma\left(A_{|(V_n^+\oplus V_n^-)}\right)\right)=0$;
\item $\lambda\notin\sigma(A)$.
\end{enumerate}
We denote by $\Spu(A,P)$ the set of $P$-spurious eigenvalues of the operator $A$.
\end{definition}

Now we will show as announced that contrarily to $\overline{\Spu(A)}\cup\hat{\sigma}_{\rm ess}(A)$ which is always an interval, $\overline{\Spu(A,P)}\cup\hat{\sigma}_{\rm ess}(A)$ is the union of two intervals, hence it may have a ``hole''. 

\begin{theorem}[Characterization of $P$-spurious eigenvalues]\label{thm:P} Let $A$ be a self-adjoint operator with dense domain $D(A)$. Let $P$ be an orthogonal projector on $\gH$ such that $P\cC\subset D(A)$ for some $\cC\subset D(A)$ which is a core for $A$. We assume that $PAP$ (resp. $(1-P)A(1-P)$) is essentially self-adjoint on $P\cC$ (resp. $(1-P)\cC$), with closure denoted as $A_{|P\gH}$ (resp. $A_{|(1-P)\gH}$).
We assume also that
\begin{equation}
\inf\hat{\sigma}_{\rm ess}\big(A_{|(1-P)\gH}\big)\leq \inf\hat{\sigma}_{\rm ess}\big(A_{|P\gH}\big).
\label{assump_ordre} 
\end{equation}
Then we have
\begin{multline}
\overline{\Spu(A,P)}\cup\hat{\sigma}_{\rm ess}(A)=\left[\inf\hat{\sigma}_{\rm ess}(A),\sup\hat{\sigma}_{\rm ess}\big(A_{|(1-P)\gH}\big)\right]\\
\cup \left[\inf\hat{\sigma}_{\rm ess}\big(A_{|P\gH}\big),\sup\hat{\sigma}_{\rm ess}(A)\right].
\label{formula__conv_spectrum}
\end{multline}
\end{theorem}

\begin{figure}[h]
\centering
\includegraphics[width=10cm]{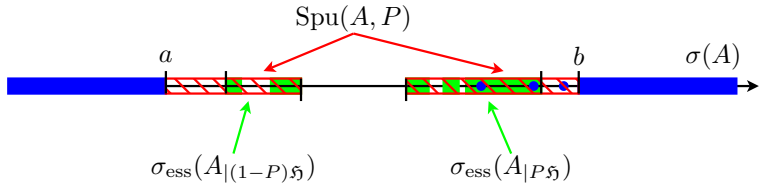}
\label{fig:poll_P}
\caption{Illustration of Theorem \ref{thm:P}: for an operator $A$ with a gap $[a,b]$ in its essential spectrum, pollution can occur in the whole gap, except between the convex hulls of $\hat{\sigma}_{\rm ess}\big(A_{|P\gH}\big)$ and $\hat{\sigma}_{\rm ess}\big(A_{|(1-P)\gH}\big)$.}
\end{figure}

Let us emphasize that condition \eqref{assump_ordre} always holds true, exchanging $P$ and $1-P$ if necessary. Usually we will assume for convenience that $1-P$ is ``associated with the lowest part of the spectrum'' in the sense of \eqref{assump_ordre}.

As mentioned before, an interesting example is when $A$ possesses a gap $[a,b]$ in its essential spectrum, i.e. such that $(a,b)\cap\sigma_{\rm ess}(A)=\emptyset$ and
$$\tr\left(\chi_{(-\ii,a]}(A)\right)=\tr\left(\chi_{[b,\ii)}(A)\right)=+\ii.$$
Then taking $\Pi=\chi_{[c,\ii)}(A)$ and $\cC=D(A)$ we easily see that 
${\Spu(A,\Pi)}\cap(a,b)=\emptyset$.
The idea that we shall pursue in the next section is simply that if $P$ is ``not too far from $\Pi$'', then we may be able to avoid completely pollution in the gap $[a,b]$. 

Before writing the proof of Theorem \ref{thm:P}, we make some remarks.

\begin{remark}\label{rmk:ess_self-adjoint}\it
If the symmetric operators $PAP$ and $(1-P)A(1-P)$ are both semi-bounded on their respective domains $P\cC$ and $(1-P)\cC$, then the inclusion $\subseteq$ in \eqref{formula__conv_spectrum} is also true provided that $A_{|P\gH}$ and $A_{|(1-P)\gH}$ are defined as the corresponding Friedrichs extensions. The essential self-adjointness is only used to show the converse inclusion $\supseteq$.
\end{remark}

\begin{remark}
An interesting consequence of Theorem \ref{thm:P} is that the set of spurious eigenvalues varies continuously when the projector $P$ is changed (in an appropriate norm for which the spectra of $A_{|P\gH}$ and $A_{|(1-P)\gH}$ change continuously). This has important practical consequences: even if one knows a projector which does not create pollution, it could in principle be difficult to numerically build a basis respecting the splitting of $\gH$ induced by $P$. However we know that pollution will only appear at the edges of the gap if the elements of the Galerkin basis are only known approximately.
\end{remark}

\begin{proof}
We will make use of the following result, whose proof will be omitted (it is an obvious adaptation of the proof of Lemma \ref{lem:reciproque_Weyl}):
\begin{lemma}\label{lem:reciproque_Weyl_P} 
Assume that $A$ is as above.
Let $(x_n^1,...,x_n^K)$ and $(y_n^1,...,y_n^{K'})$ be two orthonormal systems\footnote{We will allow $K=0$ or $K'=0$.} in $P\gH\cap D(A)$ and $(1-P)\gH\cap D(A)$ respectively, such that $x_n^j\wto0$ and $y_n^{k'}\wto0$ for all $j=1..K$ and $j'=1..K'$. Denote by $W_n$ the space spanned by $x_n^1,..., x_n^K,y_n^1,...,y_n^{K'}$. If $\lambda\in\R$ is such that 
$\lim_{n\to\ii}{\rm dist}\left(\lambda\,,\,\sigma(A_{|W_n})\right)=0,$
then $\lambda\in{\rm Spu}(A,P)\cup\sigma(A)$.
\end{lemma}

In the rest of the proof, we denote $[a,b]:=\Conv\left(\hat{\sigma}_{\rm ess}(A)\right)$, $[c_1,d_1]:={\rm Conv}\left(\hat{\sigma}_{\rm ess}\big(A_{|(1-P)\gH}\big)\right)$ and $[c_2,d_2]:={\rm Conv}\left(\hat{\sigma}_{\rm ess}\big(A_{|P\gH}\big)\right)$. For simplicity we also introduce $c=\min(c_1,c_2)=c_1$, and $d=\max(d_1,d_2)$. Recall that we have assumed $c_1\leq c_2$.

\paragraph{Step 1.}
First we collect some easy facts. The first is to note that $\Spu(A,P)\subset\Spu(A)\subset[a,b]$, where we have used Theorem \ref{thm:general}.
Next we claim that 
\begin{equation}
 [c_1,d_1]\cup[c_2,d_2]\subset{\Spu(A,P)}\cup\sigma(A)\cap[a,b]. 
\label{inclusion_pollution_1}
\end{equation}
This is indeed an obvious consequence of Theorem \ref{thm:general} applied to $A_{|P\gH}$ and $A_{|(1-P)\gH}$, and of Lemma \ref{lem:reciproque_Weyl_P}.

\paragraph{Step 2.} The second step is less obvious, it consists in proving that
\begin{equation}
[a,c]\cup[d,b]\subset{\Spu(A,P)}\cup\sigma(A)\cap[a,b]
\label{inclusion_pollution_2}
\end{equation}
which then clearly implies
$$[a,d_1]\cup[c_2,b]\subset{\Spu(A,P)}\cup\sigma(A)\cap[a,b].$$

Let us assume for instance that $d<b$ and prove the statement for $[d,b]$ (the proof is the same for $[a,c]$). Note that $b$ may \emph{a priori} be equal to $+\ii$ but of course we always have under this assumption $d<+\ii$. In principle we could however have $d=-\ii$. 
In the rest of the proof of \eqref{inclusion_pollution_2}, we fix some finite $\lambda\in(d,b)$ and prove that $\lambda\in {\Spu(A,P)}\cup\sigma(A)$. We also fix some finite $d'$ such that $d<d'<\lambda$.
We will use the following
\begin{lemma}\label{lem:Weyl_weak_0_basis} Assume that $b\in\hat\sigma_{\rm ess}(A)$. Then there exists a Weyl sequence $\{x_n\}\subset \cC$ such that $(A-b_n)x_n\to0$, $\norm{x_n}=1$, $x_n\wto0$, $b_n\to b$ and
\begin{equation}
\frac{Px_n}{\|Px_n\|}\wto0\ \text{ and }\ \frac{(1-P)x_n}{\|(1-P)x_n\|}\wto0\ \text{ weakly.}
\end{equation}
\end{lemma}
\begin{proof}
Let $b_n\to b$ and $\{y_n\}\subset \cC$ be a Weyl sequence such that $(A-b_n)y_n\to0$ with $\norm{y_n}=1$, $y_n\wto0$ (note we may assume $\{y_n\}\subset\cC$ since $\cC$ is a core for $A$). We denote $y_n=y_n^++y_n^-$ where $y_n^+\in P\cC\subset D(A)$ and $y_n^-\in P\cC\subset D(A)$. Extracting a subsequence, we may assume that $\norm{y_n^+}^2\to\ell^+$ and that  $\norm{y_n^-}^2\to\ell^-$; note $\ell^++\ell^-=1$. It is clear that if $\ell^\pm>0$, then $y_n^\pm\norm{y_n^\pm}^{-1}\wto0$ since $y_n^\pm\wto0$.
We will assume for instance $\ell^+=0$ and $\ell^-=1$. 

Next we fix an orthonormal basis $\{e_i\}\subset P\cC$ of $P\gH$, we define
$$r_k^+:=\sum_{i=1}^k\pscal{e_i,y_{n_k}}e_i$$
and note that
$$(A-b_{n_k})r_k^+=\sum_{i=1}^k\Big(\pscal{e_i,y_{n_k}}Ae_i+\pscal{e_i,(A-b_{n_k})y_{n_k}}e_i-\pscal{Ae_i,y_{n_k}}e_i\Big).$$
For $k$ fixed and any $i=1..k$, we have 
$$\lim_{n\to\ii}\pscal{e_i,y_{n}}=\lim_{n\to\ii}\pscal{e_i,(A-b_{n})y_{n}}=\lim_{n\to\ii}\pscal{Ae_i,y_{n}}=0.$$
Hence, for a correctly chosen subsequence $\{y_{n_k}^+\}$, we may assume that
$$\quad\text{ satisfies }\quad \lim_{k\to\ii}\norm{r_k^+}=\lim_{k\to\ii}\norm{(A-b_{n_k})r_k^+}=0.$$
Next we define $x_k:=y_{n_k}-r_k^+=(y_{n_k}^+-r_k^+)+y_{n_k}^-$ which satisfies $\norm{x_k}=1+o(1)$ since $\norm{r_k^+}\to0$.
By construction, we have $x_k^+=y_{n_k}^+-r_k^+\in {\rm span}(e_1,...,e_k)^\perp$, hence necessarily $x_k^+\norm{x_k^+}^{-1}\wto0$. Eventually, we have $(A-b_{n_k})x_k\to0$ strongly, by construction of $r_k^+$.
\end{proof}

In the rest of the proof we choose a sequence $\{x_n\}$ like in Lemma \ref{lem:Weyl_weak_0_basis} and denote $x_n^+=Px_n$ and $x_n^-=(1-P)x_n$. By the definition of $d$ and the fact that $A_{|(1-P)\gH}$ is essentially selfadjoint on $(1-P)\cC$, we can choose a Weyl sequence $\{y_n^-\}\subset (1-P)\cC$ such that $(1-P)(A-d_n)y_n^-\to0$, $\norm{y_n^-}=1$, $y_n^-\wto0$ weakly and $d_n\to d_1\leq d$. Extracting a subsequence from $\{y_n^-\}$ we may also assume that $y_n^-$ satisfies
\begin{equation}
\lim_{n\to\ii}\pscal{\frac{x_n^-}{\norm{x_n^-}},y_n^-}=\lim_{n\to\ii}\pscal{\frac{Ax_n^+}{\norm{x_n^+}},y_n^-}=\lim_{n\to\ii}\pscal{\frac{Ax_n^-}{\norm{x_n^-}},y_n^-}=0
\label{limit_pscal} 
\end{equation}

Let us now introduce the following orthonormal system
\begin{equation}
\left(\frac{x_n^+}{\norm{x_n^+}}\, ,\, v_n(\theta)\right)\qquad \text{ with }\quad v_n(\theta):=\frac{\cos\theta\;\frac{x_n^-}{\norm{x_n^-}}+\sin\theta\; y_n^-}{\sqrt{1+2\Re\cos\theta\sin\theta\pscal{\frac{x_n^-}{\norm{x_n^-}},y_n^-}}}
\label{def_basis} 
\end{equation}
and denote by $A_n(\theta)$ the $2\times2$ matrix of $A$ in this basis, with eigenvalues $\lambda_n(\theta)\leq\mu_n(\theta)$. As $x_n^+\norm{x_n^+}^{-1}\wto0$ weakly, we have 
\begin{equation}
\limsup_{n\to\ii}\sup_{\theta\in[0,\pi/2]}\lambda_n(\theta)\leq \limsup_{n\to\ii}\frac{\pscal{Ax_n^+,x_n^+}}{\norm{x_n^+}^2}\leq d_2\leq d.
\label{estim_first_eigenval} 
\end{equation}
When $\theta=0$, we know by construction of $x_n$ that $A_n(0)$ has an eigenvalue which converges to $b$ as $n\to\ii$. Since $b>d$ by assumption, this shows by \eqref{estim_first_eigenval} that this eigenvalue must be $\mu_n(0)$, hence we have $\mu_n(0)\to b$ as $n\to\ii$. On the other hand, the largest eigenvalue of $A_n(\pi/2)$ satisfies for $n$ large enough
$$\mu_n(\pi/2)\leq \max\left(\frac{\pscal{Ax_n^+,x_n^+}}{\norm{x_n^+}^{2}}\;,\; \pscal{Ay_n^-,y_n^-}\right)+\left|\pscal{\frac{Ax_n^+}{\norm{x_n^+}},y_n^-}\right|\leq d',$$
where we have used \eqref{limit_pscal}, $x_n^+\norm{x_n^+}^{-1}\wto0$, $y_n^-\wto0$, and the definition of $d'>d$. 

By continuity of $\mu_n(\theta)$, there exists a $\theta_n\in(0,\pi/2)$ such that $\mu_n(\theta_n)=\lambda$. 
Next we note that the two elements of the basis defined in \eqref{def_basis} both go weakly to zero by the construction of $x_n^\pm$ and of $y_n^-$. Hence our statement $\lambda\in {\Spu(A,P)}\cup\sigma(A)$ follows from Lemma \ref{lem:reciproque_Weyl_P}. 

\paragraph{Step 3.} The last step is to prove that when $d_1<c_2$, 
$$(d_1,c_2)\cap\big(\Spu(A,P)\cup\sigma_{\rm ess}(A)\big)=\emptyset$$
(there is nothing else to prove when $c_2\leq d_1$). 
We will prove that $(d_1,c_2)\cap\Spu(A,P)=\emptyset$, the proof for $\sigma_{\rm ess}(A)$ being similar. Note that under our assumption $d_1<c_2$, we must have $d_1<\ii$ and $c_2>-\ii$, hence $A_{|P\gH}$ and $A_{|(1-P)\gH}$ are semi-bounded operators.
As noticed in Remark \ref{rmk:ess_self-adjoint}, it is sufficient to assume for this step that $A_{|P\gH}$ and $A_{|(1-P)\gH}$ are the Friedrichs extensions of $(PAP,P\cC)$ and $((1-P)A(1-P),(1-P)\cC)$ without assuming \emph{a priori} that they are essentially self-adjoint.

Now we argue by contradiction and assume that there exists a Weyl sequence $\{x_n\}\in V_n^+\oplus V_n^-\subset D(A)$ like in Lemma \ref{lem:Weyl}, for some $\lambda\in(d_1,c_2)$. We will write $x_n=x_n^++x_n^-$ with $x_n^+\in V_n^+$ and $x_n^-\in V_n^-$. We have $P_{|V_n^+\oplus V_n^-}(A-\lambda)x_n\to0$, hence taking the scalar product with $x_n^+$ and $x_n^-$, we obtain
\begin{equation}
\lim_{n\to\ii}\pscal{(A-\lambda)x_n,x_n^+}=\lim_{n\to\ii}\pscal{(A-\lambda)x_n,x_n^-}=0.
\label{limit_pscal_pm}
\end{equation}
The space $\cC$ being a core for $A$, it is clear that we may assume further that $x_n\in\cC$ and still that \eqref{limit_pscal_pm} holds true. In this case we have $x_n^+,x_n^-\in D(A)$ hence we are allowed to write
$$\pscal{(A-\lambda)x_n^+,x_n^+}+\pscal{(A-\lambda)x_n^-,x_n^+}\to0,$$
$$\pscal{(A-\lambda)x_n^-,x_n^-}+\overline{\pscal{(A-\lambda)x_n^-,x_n^+}}\to0.$$
Taking the complex conjugate of the second line (the first term is real since $A$ is self-adjoint) and subtracting the two quantities, we infer that
\begin{equation}
\pscal{(A-\lambda)x_n^+,x_n^+}-\pscal{(A-\lambda)x_n^-,x_n^-}\to0.
\label{estim_convex_hull}
\end{equation}
As by assumption $\lambda\in(d_1,c_2)$, we have as quadratic forms on $P\cC$ and $(1-P)\cC$, $P(A-\lambda)P\geq \epsilon P-r$ and $-(1-P)(A-\lambda)(1-P)\geq \epsilon(1-P)-r'$ for some finite-rank operators $r$ and $r'$ and some $\epsilon>0$ small enough. 
Hence we have 
$$\pscal{(A-\lambda)x_n^+,x_n^+}-\pscal{(A-\lambda)x_n^-,x_n^-}\geq \epsilon\norm{x_n^+}^2+\epsilon\norm{x_n^-}^2+o(1).$$
This shows that we must have $x_n\to0$ which is a contradiction. \hfill$\square$
\end{proof}

\subsection{A simple criterion of no pollution}\label{sec:splitting_condition}
Here we give a very intuitive condition allowing to avoid pollution in a gap.
\begin{theorem}[Compact perturbations of spectral projector do not pollute]\label{thm:cond_no_pollute}
Let $A$ be a self-adjoint operator defined on a dense domain $D(A)$, and let $a<b$ be such that 
\begin{equation}
(a,b)\cap\sigma_{\rm ess}(A)=\emptyset\ \text{ and }\  \tr\left(\chi_{(-\ii,a]}(A)\right)=\tr\left(\chi_{[b,\ii)}(A)\right)=+\ii.
\label{assumption_a_b}
\end{equation}
Let $c\in(a,b)\setminus \sigma(A)$ and denote $\Pi:=\chi_{(c,\ii)}(A)$.
Let $P$ be an orthogonal projector satisfying the assumptions of Theorem \ref{thm:P}. We furthermore assume that 
$(P-\Pi)|A-c|^{1/2}$, initially defined on $D(|A-c|^{1/2})$, extends to a compact operator on $\gH$. Then we have
$$\Spu(A,P)\cap(a,b)=\emptyset.$$
\end{theorem}

As we will see in Corollary \ref{cor:no_pollute}, Theorem \ref{thm:cond_no_pollute} is useful when our operator takes the form $A+B$ where $B$ is $A$-compact. Using the spectral projector $P=\Pi$ of $A$ will then avoid pollution for $A+B$, when $A$ is bounded from below.

\begin{remark}\label{contre_exemple1}\it
We give an example showing that the power $1/2$ in $|A-c|^{1/2}$ is sharp. Consider for instance an orthonormal basis $\{e_n^\pm\}$ of a separable Hilbert space $\gH$, and define 
$A:=\sum_{n\geq1}n|e_n^+\rangle\langle e_n^+|.$
Choosing $c=1/2$, we get 
$\Pi=\chi_{[1/2,\ii)}(A)=\sum_{n\geq1}|e_n^+\rangle\langle e_n^+|.$
Define now a new basis by $f_n^+=\cos\theta_ne_n^++\sin\theta_ne_n^-$, $f_n^-=\sin\theta_ne_n^+-\cos\theta_ne_n^-$, and introduce the associated projector $P=\sum_{n\geq1}|f_n^+\rangle\langle f_n^+|$. Consider then $V_n:=\text{span}\{f_1^\pm,\cdots,f_{n-1}^\pm,f_n^-\}$ for which we have $\sigma(A_{|V_n})=\{0,1,\cdots,n-1,n\sin^2\theta_n\}$. On the other hand it is easily checked that $(P-\Pi)|A-1/2|^\alpha$ is compact if and only if $n^\alpha \theta_n\to0$ as $n\to\ii$. Hence, if $0\leq \alpha<1/2$ we can take $\theta_n=1/\sqrt{2n}$ and we will have a polluted eigenvalue at $1/2$ whereas $(P-\Pi)|A-1/2|^\alpha$ is compact.
\end{remark}

We now write the proof of Theorem \ref{thm:cond_no_pollute}.

\begin{proof}
We will prove that $\sigma_{\rm ess}(A_{|P\gH})\subset[b,\ii)$. This will end the proof, by Theorem \ref{thm:P} and a similar argument for $A_{|(1-P)\gH}$. Assume on the contrary that $\lambda\in(-\ii,b)\cap\sigma_{\rm ess}(A_{|P\gH})$. Without any loss of generality, we may assume that $c>\lambda$ (changing $c$ if necessary). As $P\cC$ is a core for $A_{|P\gH}$, there exists a sequence $\{x_n\}\subset P\cC$ such that $x_n\wto0$ weakly in $\gH$, $\norm{x_n}=1$ and $P(A-\lambda)x_n\to0$ strongly in $\gH$. We have
\begin{multline}
 \pscal{(P-\Pi)(A-\lambda)(P-\Pi)x_n,x_n}
+2\Re\pscal{(P-\Pi)(A-c)\Pi x_n,x_n}\\+\pscal{\Pi(A-\lambda)\Pi x_n,x_n}=\pscal{P(A-\lambda)x_n,x_n}+(\lambda-c)2\Re\pscal{\Pi x_n,(P-\Pi)x_n}\label{decomp_proj}
\end{multline}
where we note that $Px_n=x_n\in D(A)$ and $\Pi x_n\in D(A)$ since $\Pi$ stabilizes $D(A)$.
As $c\notin\sigma(A)$, we have that $|A-c|^{-1/2}$ is bounded, hence $P-\Pi$ must be a compact operator, i.e. the last term of the right hand side of \eqref{decomp_proj} tends to 0 as $n\to\ii$.
By the Cauchy-Schwarz inequality we have
\begin{equation}
|\pscal{(P-\Pi)(A-c)\Pi x_n,x_n}|\leq \norm{|A-c|^{1/2}\Pi x_n}\norm{|A-c|^{1/2}(P-\Pi) x_n}.
\end{equation}
As by assumption $(P-\Pi)|A-c|^{1/2}$ is compact, we have that $(P-\Pi)(A-\lambda)(P-\Pi)$ and $|A-c|^{1/2}(P-\Pi)$ are also compact operators. Hence 
$$\lim_{n\to\ii}\norm{|A-c|^{1/2}(P-\Pi) x_n}=\lim_{n\to\ii}\pscal{(P-\Pi)(A-\lambda)(P-\Pi)x_n,x_n}=0.$$
On the other hand we have $\Pi(A-\lambda)\Pi=\Pi(A-c)\Pi+(c-\lambda)\Pi\geq \Pi|A-c|\Pi$ since we have chosen $c$ in such a way that $c>\lambda$, and by the definition of $\Pi$. Hence by \eqref{decomp_proj} we have an inequality of the form
$$\norm{|A-c|^{1/2}\Pi x_n}^2-2\epsilon_n\norm{|A-c|^{1/2}\Pi x_n}\leq \epsilon'_n$$
where $\lim_{n\to\ii}\epsilon_n=\lim_{n\to\ii}\epsilon'_n=0$. This clearly shows that 
$$\lim_{n\to\ii}\norm{|A-c|^{1/2}\Pi x_n}=0.$$
Therefore we deduce $\Pi x_n\to0$ strongly, $|A-c|^{1/2}$ being invertible. Hence $x_n=Px_n=(P-\Pi)x_n+\Pi x_n\to0$ and we have reached a contradiction.
\end{proof}

We now give a simple application of the above result.
\begin{corollary}\label{cor:no_pollute}
Let $A$ be a \underline{bounded-below} self-adjoint operator defined on a dense domain $D(A)$, and let $a<b$ be such that 
\begin{equation}
(a,b)\cap\sigma_{\rm ess}(A)=\emptyset\ \text{ and }\  \tr\left(\chi_{(-\ii,a]}(A)\right)=\tr\left(\chi_{[b,\ii)}(A)\right)=+\ii.
\label{assumption_a_b_bis}
\end{equation}
Let $c\in(a,b)$ be such that $c\notin \sigma(A)$ and denote $\Pi:=\chi_{(c,\ii)}(A)$.

Let $B$ be a symmetric operator such that $A+B$ is self-adjoint on $D(A)$ and such that $\big((A+B-i)^{-1}-(A-i)^{-1}\big)|A-c|^{1/2}$, initially defined on $D(|A-c|^{1/2})$, extends to a compact operator on $\gH$. Then we have
$$\Spu(A+B,\Pi)\cap(a,b)=\emptyset.$$
\end{corollary}
\begin{proof}
Under our assumption we have that $(A+B-i)^{-1}-(A-i)^{-1}$ is compact, hence $\sigma_{\rm ess}(A+B)=\sigma_{\rm ess}(A)$ by Weyl's Theorem \cite{ReeSim4,Davies} and $A+B$ is also bounded from below. Changing $c$ if necessary we may assume that $c\notin\sigma(A+B)\cup\sigma(A)$. Next we take a curve $\curlC$ in the complex plane enclosing the whole spectrum of $A$ and $A+B$ below $c$ (i.e. intersecting the real axis only at $c$ and $c'<\inf\sigma(A)\cup\sigma(A+B)$). In this case, we have by Cauchy's formula and the resolvent identity
\begin{multline*}
\bigg(\Pi-\chi_{[c,\ii)}(A+B)\bigg)|A-c|^{1/2}=-\frac1{2i\pi}\oint_{\curlC}\left(\frac{1}{A+B-z}-\frac{1}{A-z}\right)|A-c|^{1/2}dz\\
=-\frac1{2i\pi}\oint_{\curlC}\frac{A+B-i}{A+B-z}\left(\frac{1}{A+B-i}-\frac{1}{A-i}\right)|A-c|^{1/2}\frac{A-i}{A-z}dz
\end{multline*}
Since $\curlC$ is bounded (we use here that $A$ is bounded-below), we easily deduce that the above operator is compact, hence the result follows from Theorem \ref{thm:cond_no_pollute}.
\end{proof}

 \begin{remark}\it
Again the power $1/2$ in $|A-c|^{1/2}$ is optimal, as seen by taking  $B=-A+\sum_{n}n|f_n^+\rangle\langle f_n^+|$ where $A$, $f_n^+$ and $\theta_n$ are chosen as in Remark  \ref{contre_exemple1} and $V_n:=\{e_1^\pm,...,e_{n-1}^\pm,e_n^-\}$.
 \end{remark}

 \begin{remark}\it
Corollary \ref{cor:no_pollute} is \emph{a priori} wrong when $A$ is not semi-bounded. This is seen by taking for instance $A=\sum_{n\geq1}n|e_n^+\rangle\langle e_n^+|-\sum_{n\geq1}n|e_n^-\rangle\langle e_n^-|$ and $B=-A+\sum_{n\geq1}n|f_n^+\rangle\langle f_n^+|-\sum_{n\geq1}n|f_n^-\rangle\langle f_n^-|$ where $f_n^+=e_n^+/\sqrt{2}+e_n^-/\sqrt{2}$ and $f_n^-=-e_n^-/\sqrt{2}+e_n^+/\sqrt{2}$. A short calculation shows that $\big((A+B)^{-1}-A^{-1}\big)|A|^{\alpha}$ is compact for all $0\leq\alpha<1$ whereas $0\in\Spu(A+B,\Pi)$ which is seen by choosing again $V_n=\{e_1^\pm,...,e_{n-1}^{\pm},e_n^-\}$.
 \end{remark}

\subsection{Applications}\label{sec:applications_splitting}

\subsubsection{Periodic Schrödinger operators in Wannier basis}\label{sec:Periodic}
In this section, we show that approximating eigenvalues in gaps of periodic Schrödinger operators using a so-called \emph{Wannier basis} does not yield any spectral pollution. This method was already successfully applied in dimension 1 in \cite{CanDelLew-08b} for a nonlinear model introduced in \cite{CanDelLew-08b}. For references on pollution in this setting, we refer for example to \cite{BouLev-07}.

Consider $d$ linearly independent vectors $a_1,...,a_d$ in $\R^d$ and denote by 
$$\curlL:=a_1\Z\oplus\cdots \oplus a_d\Z$$
the associated lattice. We also define the \emph{dual lattice} 
$$\curlL^*:=a^*_1\Z\oplus\cdots \oplus a^*_d\Z\ \text{ with }\ \pscal{a_i,a^*_j}=(2\pi)\delta_{ij}.$$
Finally, the Brillouin zone is defined by
$$\curlB:=\left\{x\in\R^d\ |\ \norm{x}=\inf_{k\in\curlL^*}\norm{x-k}\right\}.$$
Next we fix an $\curlL$-periodic potential $V_{\rm per}$, i.e. $V_{\rm per}(x+a)=V_{\rm per}(x)$ for all $a\in\curlL$. We will assume as usual \cite{ReeSim4} that
$$V_{\rm per}\in L^p(\curlB)\ \text{ where }\left\{\begin{array}{ll}
p=2 &\text{ if }d\leq3,\\
p>2 &\text{ if }d=4,\\
p=d/2 &\text{ if }d\geq5.
\end{array}\right.$$
In this case it is known \cite{ReeSim4} that the operator
\begin{equation}
A_{\rm per}=-\Delta+V_{\rm per}
\label{def_A_per}
\end{equation}
is self-adjoint on $H^2(\R^d)$. One has the \emph{Bloch-Floquet decomposition}
$$A_{\rm per}=\frac{1}{|\curlB|}  \int_{\curlB}^\oplus A_{\rm per}(\xi) \, d\xi$$
where $A_{\rm per}(\xi)$ is for almost all $\xi\in \curlB$ a self-adjoint operator acting on the space
$$L^2_\xi = \left\{ u \in L^2_{\rm   loc}(\R^3)\ |\ u(x+a) = e^{-i a \cdot \xi} u(x), \; \forall a \in \curlL \right\}.$$
For any $\xi$, the spectrum of $A_{\rm per}(\xi)$ is composed of a (nondecreasing) sequence of eigenvalues of finite multiplicity $\lambda_k(\xi)\nearrow\ii$, hence the spectrum
$$\sigma(A_{\rm per})=\sigma_{\rm ess}(A_{\rm per})=\bigcup_{k\geq1}\lambda_k(\curlB)$$
is composed of bands. The eigenvalues $\lambda_k(\xi)$ are known to be real-analytic in any fixed direction when $V_{\rm per}$ is smooth enough \cite{Thomas-73,ReeSim4}, in which case the spectrum of $A_{\rm per}$ is purely absolutely continuous.

The operator \eqref{def_A_per} may be used to describe quantum electrons in a crystal. It appears naturally for noninteracting systems in which case $V_{\rm per}$ is the periodic Coulomb potential induced by the nuclei of the crystal. However operators of the form \eqref{def_A_per} also appear in nonlinear models taking into account the interaction between the electrons. In this case, the potential $V_{\rm per}$ contains an additional effective (mean-field) potential induced by the electrons themselves \cite{CatBriLio-01,CanDelLew-08b}. In the presence of an impurity in the crystal, one is led to consider an operator of the form
\begin{equation}
A=-\Delta+V_{\rm per}+W.
\label{def_A_impurity}
\end{equation}
We will assume in the following that 
$$W\in L^p(\R^d)+L^\ii_\epsilon(\R^d)\ \text{ for some }\ p>\max(d/3,1)$$
in which case $(A_{\rm per}+W-i)^{-1}-(A_{\rm per}-i)^{-1}$ is $(1-\Delta)^{-1/2}$-compact as seen by the resolvent expansion \cite{ReeSim4}, and one has
$$\sigma_{\rm ess}(A)=\sigma(A_{\rm per}).$$
However eigenvalues may appear between the bands. Intuitively, they correspond to bound states of electrons (or holes) in presence of the defect. By Theorem \ref{thm:general}, their computation may lead to pollution. For a finite elements-type basis, spectral pollution was studied in \cite{BouLev-07}.

Using the Bloch-Floquet decomposition, a spectral decomposition of the reference periodic operator $A_{\rm per}$ is easily accessible numerically. This decomposition can be used as a starting point to avoid pollution for the perturbed operator $A$. For simplicity we shall assume that the spectral decomposition of $A_{\rm per}$ is known exactly. More precisely we make the assumption that there is a gap between the $k$th and the $(k+1)$st band:
$$a:=\sup\lambda_k(\curlB)<\inf\lambda_{k+1}(\curlB):=b$$
and that the associated spectral projector
$$P_{\rm per}:=\chi_{(-\ii,c)}(A_{\rm per}),\qquad c=\frac{a+b}{2}$$
is known. The interest of this approach is the following
\begin{theorem}[No pollution for periodic Schrödinger operators]\label{thm:no_poll_Wannier_basis}
We assume $V_{\rm per}$ and $W$ are as before. Then we have
\begin{equation}
\Spu(A,P_{\rm per})\cap(a,b)=\emptyset.
\label{no_pollution_periodic}
\end{equation}
\end{theorem}
\begin{proof}
This is a simple application of Corollary \ref{cor:no_pollute}.\end{proof}

It was noticed in \cite{CanDelLew-08b} that a very natural basis respecting the decomposition associated with $P_{\rm per}$ is given by a so-called \emph{Wannier basis} \cite{Wannier-37}. Wannier functions $\{w_k\}$ are defined in such a way that $w_k$ belongs to the spectral subspace associated with the $k$th band and $\{w_k(\cdot-a)\}_{a\in\curlL}$ forms a basis of this spectral subspace. One can take 
\begin{equation}
w_k(x)=\frac{1}{|\curlB|}  \int_{\curlB}u_k(\xi,x) d\xi
\label{def_Wannier_simple} 
\end{equation}
where $u_k(\xi,\cdot)\in L^2_\xi$ is for any $\xi\in\curlB$ an eigenvector of $A_{\rm per}(\xi)$ corresponding to the $k$th eigenvalue $\lambda_k(\xi)$. The so-defined $\{w_k(\cdot-a)\}_{a\in\curlL}$ are mutually orthogonal. Formula \eqref{def_Wannier_simple} does not define $w_k$ uniquely since the $u_k(\xi,x)$ are in the best case only known up to a phase. Choosing the right phase, one can prove that when the $k$th band is isolated from other bands, $w_k$ decays exponentially \cite{Nenciu-83}. 

More generally, instead of using only one band (i.e. one eigenfunction $u_k(\xi,x)$), one can use $K$ different bands for which it is possible to construct $K$ exponentially localized Wannier functions as soon as the union of the $K$ bands is isolated from the rest of the spectrum \cite{Panati-07,Brouder-etal-07}. The union of the $K$ bands is called a \emph{composite band}.

In our case we typically have a natural composite band corresponding to the spectrum of $A_{\rm per}$ which is below $c$, and another one corresponding to the spectrum above $c$ (the latter is not bounded above). By Theorem \ref{thm:no_poll_Wannier_basis}, we know that using such a basis will not create any pollution in the gap of $A$. 

We emphasize that the Wannier basis \emph{does not} depend on the decaying potential $W$, and can be precalculated once and for all for a given $\curlL$ and a given $V_{\rm per}$. Another huge advantage is that since $w_k$ decays fast, it will be localized over a certain number of unit cells of $\curlL$. As $W$ represents a localized defect in the lattice, keeping only the Wannier functions $w_k(\cdot-a)$ with $a\in\curlL\cap B(0,R)$ for some radius $R>0$ should already yield a very good approximation to the spectrum in the gap (we assume that the defect is localized in a neighborhood of $0$). This approximation can be improved by enlarging progressively the radius $R$.

Of course in practice exponentially localized Wannier functions are not simple to calculate. But some authors have defined the concept of \emph{maximally localized Wannier functions} \cite{MarVdB-97} and proposed efficient methods to find these functions numerically.

The efficiency of the computation of the eigenvalues of $A$ in the gap using a Wannier basis (compared to that of the so-called super-cell method) were illustrated for a nonlinear model in \cite{CanDelLew-08b}.

\subsubsection{Dirac operators in upper/lower spinor basis}\label{sec:Dirac1}
The Dirac operator is a differential operator of order 1 acting on $L^2(\R^3,\C^4)$, defined as \cite{Thaller,EstLewSer-08}
\begin{equation}
D^0=-ic\sum_{k=1}^3\alpha_k\partial_{x_k}+mc^2\beta:=c\alp\cdot p+mc^2\beta.
\end{equation}
Here $\alpha_1,\alpha_2,\alpha_3$ and $\beta$ are the so-called Pauli $4\times4$ matrices \cite{Thaller} 
which are chosen to ensure that 
$$(D^0)^2=-c^2\Delta+m^2c^4.$$
The usual representation in $2\times 2$ blocks is given by 
$$ \beta=\left( \begin{matrix} I_2 & 0 \\ 0 & -I_2 \\ \end{matrix} \right),\quad \; \alpha_k=\left( \begin{matrix}
0 &\sigma_k \\ \sigma_k &0 \\ \end{matrix}\right)  \qquad (k=1, 2, 3)\,,
$$
where the Pauli matrices are defined as
\begin{equation}
\sigma _1=\left( \begin{matrix} 0 & 1
\\ 1 & 0 \\ \end{matrix} \right),\quad  \sigma_2=\left( \begin{matrix} 0 & -i \\
i & 0 \\  \end{matrix}\right),\quad  \sigma_3=\left( 
\begin{matrix} 1 & 0\\  0 &-1\\  \end{matrix}\right) \, . 
\label{def:Pauli}
\end{equation}
In the whole paper we use the common notation $p=-i\nabla$.

The operator $D^0$ is self-adjoint on $H^1(\R^3,\C^4)$ and its spectrum is symmetric with respect to zero: $\sigma(D^0)=(-\ii,-mc^2]\cup[mc^2,\ii).$
An important problem is to compute eigenvalues of operators of the form
$$D^V=D^0+V$$
in the gap $(-mc^2,mc^2)$, where $V$ is a multiplication operator by a real function $x\mapsto V(x)$. Loosely speaking, positive eigenvalues correspond to bound states of a relativistic quantum electron in the external field $V$, whereas negative eigenvalues correspond to bound states of a positron, the anti-particle of the electron. In practice, spectral pollution is an important problem \cite{DraGol-81,Grant-82,Kutzelnigg-84,StaHav-84} which is dealt with in Quantum Physics and Chemistry by means of several different methods, the most widely used being the so-called \emph{kinetic balance} which we will study later in Section \ref{sec:kinetic_balance}. We refer to \cite{BouBou-08} for a recent numerical study based on the so-called \emph{second-order method} for the radial Dirac operator.

We now present a heuristic argument which can be made mathematically rigorous in many cases \cite{Thaller,EstLewSer-08}. First we write the equation satisfied by an eigenvector $(\phi,\chi)$ of $D^0+V$ with eigenvalue $mc^2+\lambda\in (-mc^2,mc^2)$ as follows:
\begin{equation}
\left\{\begin{array}{l}
(mc^2+V)\phi+c\sigma\cdot(-i\nabla)\chi=(mc^2+\lambda)\phi,\\
(-mc^2+V)\chi+c\sigma\cdot(-i\nabla)\phi=(mc^2+\lambda)\chi,
\end{array}\right.
\label{eq:eigenvalue_upper_lower}
\end{equation}
where we recall that $\sigma=(\sigma_1,\sigma_2,\sigma_3)$ are the Pauli matrices defined in \eqref{def:Pauli}. 
Hence one deduces that (when it makes sense)
\begin{equation}
\chi=\frac{c}{2mc^2+\lambda-V}\sigma\cdot(-i\nabla)\phi.
\label{relation_upper_lower}
\end{equation}
If $V$ and $\lambda$ stay bounded, we infer that, at least formally,
\begin{equation}
\left(\begin{array}{c}
\phi\\ \chi
\end{array}\right)\sim_{c\to\ii}\left(\begin{array}{c}
\phi\\ \frac{1}{2mc}\sigma\cdot(-i\nabla)\phi
\end{array}\right).
\label{limit_upper_lower_c} 
\end{equation}
Hence we see that in the nonrelativistic limit $c\to\ii$, the eigenvectors of $A$ associated with a positive eigenvalue converge to a vector of the form $\left(\begin{matrix} \phi\\ 0\end{matrix}\right)$. Reintroducing the asymptotic formula \eqref{limit_upper_lower_c} of $\chi$ in the first equation of \eqref{eq:eigenvalue_upper_lower}, one gets that $\phi$ is an eigenvector of the nonrelativistic operator $-\Delta/(2m)+V$ in $L^2(\R^3,\C^2)$.

For this reason, it is very natural to consider a splitting of the Hilbert space $L^2(\R^3,\C^4)$ into \emph{upper} and \emph{lower spinor} and we introduce the following orthogonal projector
\begin{equation}
\boxed{\cP\left(\begin{matrix} \phi\\ \chi\end{matrix}\right)=\left(\begin{matrix} \phi\\ 0\end{matrix}\right),\qquad \phi,\chi\in L^2(\R^3,\C^2).}
\label{def_P_upper_lower_spinor}
\end{equation}
This splitting is the choice of most of the methods we are aware of in Quantum Physics and Chemistry.
Applying Theorem \ref{thm:P}, we can characterize the spurious spectrum associated with this splitting. For simplicity we take $m=c=1$ in the following.

\begin{theorem}[Pollution in upper/lower spinor basis for Dirac operators]\label{thm:upper_lower_spinors}
Assume that the real function $V$ satisfies the following assumptions:
\begin{description}
 \item[$(i)$] there exist $\{R_k\}_{k=1}^M\subset\R^3$ and a positive number $r<\inf_{k\neq\ell}|R_k-R_\ell|/2$ such that
\begin{equation}
 \max_{k=1..K}\sup_{|x-R_k|\leq r}|x-R_k|\; |V(x)|<\frac{\sqrt{3}}2;
\label{assumption_V_1}
\end{equation}
\item[$(ii)$] one has\footnote{We use the notation of \cite{ReeSim4}: $X+L^\ii_\epsilon=\{f\in X+L^\ii\ |\ \forall\epsilon>0,\ \exists f_\epsilon\in X \text{ such that } \norm{f-f_\epsilon}_{L^\ii}\leq\epsilon\}$.}
\begin{equation}
V\1_{\R^3\setminus\cup_1^KB(R_k,r)}\in L^p(\R^3)+L^\ii_\epsilon(\R^3)\qquad \text{for some } 3<p<\ii.
\label{assumption_V_2}
\end{equation}
\end{description}
Let $\cP$ be as in \eqref{def_P_upper_lower_spinor}. Then one has
\begin{equation}
\overline{\Spu(D^0+V,\cP)}
=\left\{\Conv\left({\rm Ess}\big(1+V\big)\right)\cup \Conv\left({\rm Ess}\big(-1+V\big)\right)\right\}\cap[-1,1]
\end{equation}
where ${\rm Ess}(W)$ denotes the essential range of the function $W$, i.e.
$${\rm Ess}(W)=\big\{\lambda\in\R\ |\ \big|W^{-1}([\lambda-\epsilon,\lambda+\epsilon])\big|\neq0\ \forall\epsilon>0\big\}.$$
\end{theorem}

\begin{remark}
It is known that the operator $D^0+V$ is essentially self-adjoint on $C_0^\ii(\R^3,\C^4)$ when \eqref{assumption_V_1} and \eqref{assumption_V_2} hold, and that its domain is simply the domain $H^1(\R^3,\C^4)$ of the free Dirac operator. When $\sqrt{3}/2$ is replaced by $1$ in \eqref{assumption_V_1}, the operator $D^0+V$ still has a \emph{distinguished} self-adjoint extension \cite{Thaller} whose associated domain satisfies $H^1(\R^3,\C^4)\subsetneq D(D^0+V)\subset H^{1/2}(\R^3,\C^4)$. Furthermore this domain is not stable by the projector $\cP$ on the upper spinor (a characterization of this domain was given in \cite{EstLos-07}). The generalization to this case is possible but it is outside the scope of this paper.
\end{remark}
\begin{remark}\label{rmk:contre-exemple}\it
By Theorem \ref{thm:upper_lower_spinors}, we see that ${\rm Spu}(D^0,\cP)=\emptyset$ but $\Spu(D^0+V,\cP)\neq\emptyset$ for all smooth potentials $V\neq0$ even if $V$ is $D^0$-compact. Hence spectral pollution is in general not stable under \emph{relatively compact perturbations} (but it is obviously stable under \emph{compact perturbations} as we have already mentioned in Remark \ref{rmk:compact-perturb}).
\end{remark}

Our assumptions on $V$ cover the case of the Coulomb potential, $V(x)=\kappa|x|^{-1}$ when $|\kappa|<\sqrt{3}/2$.
In our units, this corresponds to nuclei which have less than $118$ protons, which covers all existing atoms. On the other hand, a typical example for which $V\in  L^p(\R^3)\cap L^\ii(\R^3)$ is the case of smeared nuclei $V=\rho\ast1/|x|$ where $\rho$ is a (sufficiently smooth) distribution of charge for the nuclei.
We now give the proof of Theorem \ref{thm:upper_lower_spinors}:

\begin{proof}
Under assumptions $(i)$ and $(ii)$, it is known that $D^0+V$ is self-adjoint on $H^1(\R^3,\C^4)$ (which is stable under the action of $\cP$) and that $\sigma_{\rm ess}(D^0+V)=(-\ii,-1]\cup[1,\ii)$. We will simply apply Theorem \ref{thm:P} with $\cC=H^1(\R^3,\C^4)$.
We have in the decomposition of $L^2(\R^3)$ associated with $\cP$,
$$D^0+V=\left(\begin{matrix}
1+V & \sigma\cdot (-i\nabla)\\
\sigma\cdot (-i\nabla) & -1+V
\end{matrix}\right).$$ 
Hence $\cP(D^0+V)\cP=1+V$ and $(1-\cP)(D^0+V)(1-\cP)=-1+V$, both seen as operators acting on $L^2(\R^3,\C^2)$. It is clear that $D(D^0+V)\cap \cP L^2(\R^3,\C^4)\simeq H^1(\R^3,\C^2)$ is dense in the domain of the multiplication operator by $V(x)$
$$D(V)=\{f\in L^2(\R^3,\C^2)\ |\ Vf\in L^2(\R^3,\C^2)\},$$ 
for the associated norm
$$\norm{f}_{G(V)}^2=\int_{\R^3}\left(1+|V(x)|^2\right)|f(x)|^2dx.$$
Also the spectrum of $V$ is the essential range of $V$. Note under our assumptions on $V$ we have that $0\in{\rm Ess}(V)$. The rest follows from Theorem \ref{thm:P}.
\end{proof}

\subsubsection{Dirac operators in dual basis}\label{sec:Dirac_dual}
In this section we study a generalization of the decomposition into upper and lower spinors, which was introduced by Shabaev et al \cite{Shaetal-04}. For any fixed $\epsilon$, we consider the unitary operator
\begin{equation}
U_\epsilon:=\frac{D^0(\epsilon p)}{|D^0(\epsilon p)|} 
\label{def_U_epsilon}
\end{equation}
which is just a dilation of the sign of $D^0$ (note that $(U_\epsilon)^*=U_\epsilon$). Next we define the following orthogonal projector
\begin{equation}
\cP_\epsilon:=U_\epsilon \cP U_\epsilon
\label{def_P_epsilon} 
\end{equation}
where $\cP$ is the projector on the upper spinors as defined in \eqref{def_P_upper_lower_spinor}. As for $\epsilon=0$ we have $U_0=1$, we deduce that $\cP_0=\cP$. However, as we will see below, the limit $\epsilon\to0$ seems to be rather singular from the point of view of spectral pollution. We note that any vector in $\cP_\epsilon L^2(\R^3,\C^4)$ may be written in the following simple form
$$\left(\begin{array}{c}
\phi\\ \epsilon\sigma\cdot(-i\nabla)\phi
\end{array}\right)\quad\text{with}\quad \phi\in H^1(\R^3,\C^2).$$
Hence for $\epsilon\ll1$, the above choice just appears as a kind of correction to the simple decomposition into upper and lower spinors. Also we notice that $\cP_\epsilon H^1(\R^3,\C^4)\subset H^1(\R^3,\C^4)$ for every $\epsilon$ since $U_\epsilon$ is a multiplication operator in Fourier space and $\cP$ stabilizes $H^1(\R^3,\C^4)$.

In \cite{Shaetal-04}, the projector $\cP_\epsilon$ is considered with $\epsilon=1/(2mc)$ as suggested by Equation \eqref{limit_upper_lower_c}. However here we will for convenience let $\epsilon$ free. The method was called ``\emph{dual}'' in \cite{Shaetal-04} since contrarily to the ones that we will study later on (the \emph{kinetic} and \emph{atomic balance} methods), the two subspaces $\cP_\epsilon L^2(\R^3,\C^4)$ and $(1-\cP_\epsilon) L^2(\R^3,\C^4)$ play a symmetric role. For this reason, the dual method was suspected to avoid pollution in the whole gap and not only in the upper part. Our main result is the following (let us recall that $m=c=1$):
\begin{theorem}[Pollution in dual basis]\label{thm:no_poll_dual_basis}
Assume that the real function $V$ satisfies the following assumptions:
\begin{description}
 \item[$(i)$] there exist $\{R_k\}_{k=1}^M\subset\R^3$ and a positive number $r<\inf_{k\neq\ell}|R_k-R_\ell|/2$ such that
\begin{equation}
 \max_{k=1..K}\sup_{|x-R_k|\leq r}|x-R_k|\; |V(x)|<\frac{\sqrt{3}}2;
\label{assumption_V_1_bis}
\end{equation}
\item[$(ii)$] one has
\begin{equation}
V\1_{\R^3\setminus\cup_1^KB(R_k,r)}\in L^p(\R^3)\cap L^\ii(\R^3)\qquad \text{for some } 3<p<\ii.
\label{assumption_V_2_bis}
\end{equation}
\end{description}
Let $0<\epsilon\leq1$ and $\cP_\epsilon$ as defined in \eqref{def_P_epsilon}. Then one has\footnote{Recall that $[a,b]=\emptyset$ if $b<a$.}
\begin{equation}
\overline{\Spu(D^0+V,\cP_\epsilon)}
=\left[-1\; ,\; \min\left\{-\frac2\epsilon+1+\sup V\;,\;1\right\}\right]\cup\left[\max\left\{-1\;,\;\frac2\epsilon-1+\inf V\right\}\; ,\; 1\right].
\label{Spu_P_epsilon}
\end{equation}
\end{theorem}

Our result shows that contrarily to the decomposition into upper and lower spinors studied in the previous section, the use of $\cP_\epsilon$ indeed allows to avoid spectral pollution under the condition that $V$ is a bounded potential and that $\epsilon$ is small enough:
$$\epsilon\leq\frac{2}{2+|V|}.$$
This mathematically justifies a claim of \cite{Shaetal-04}.
However we see that for Coulomb potentials, we will again get pollution in the whole gap, independently of the choice of $\epsilon$. Also for large but bounded potentials (like the ones approximating a Coulomb potential), one might need to take $\epsilon$ so small  that this could give rise to a numerical instability.

\begin{proof}
We will again apply Theorem  \ref{thm:P}. 
We choose $\cC=U_\epsilon C_0^\ii(\R^3,\C^4)$. Note that $\cC$ is a core for $D^0+V$ (its domain is simply $H^1(\R^3,\C^4)$) and that $\cP_\epsilon\cC\subset \cap_{s\geq0}H^s(\R^3,\C^4)$ since $\cP_\epsilon$ and $U_\epsilon$ commute with the operator $p=-i\nabla$. An easy computation yields
\begin{multline}
U_\epsilon(D^0+V)_{|\cP_\epsilon\cC}U_\epsilon\simeq 1+\frac{1}{\sqrt{1+\epsilon^2|p|^2}}V\frac{1}{\sqrt{1+\epsilon^2|p|^2}}\\
+\frac{\epsilon\sigma\cdot p}{\sqrt{1+\epsilon^2|p|^2}}\left(\frac2\epsilon-2+V\right)\frac{\epsilon\sigma\cdot p}{\sqrt{1+\epsilon^2|p|^2}}:=A_1 \label{def_op_DKB1}
\end{multline}
and
\begin{multline}
U_\epsilon(D^0+V)_{|(1-\cP_\epsilon)\cC}U_\epsilon\simeq-1+\frac{1}{\sqrt{1+\epsilon^2|p|^2}}V\frac{1}{\sqrt{1+\epsilon^2|p|^2}}\\
+\frac{\epsilon\sigma\cdot p}{\sqrt{1+\epsilon^2|p|^2}}\left(-\frac2\epsilon+2+V\right)\frac{\epsilon\sigma\cdot p}{\sqrt{1+\epsilon^2|p|^2}}:=A_2.\label{def_op_DKB2}
\end{multline}
Strictly speaking these operators should be defined on $\cP U_\epsilon\cC$ and $(1-\cP)U_\epsilon\cC$ but we have made the identification $\cP U_\epsilon\cC\simeq(1-\cP)U_\epsilon\cC\simeq C^\ii_0(\R^3,\C^2)$. Let us remark that for $\epsilon>0$ the term $K:=(1+\epsilon^2|p|^2)^{-1/2}V(1+\epsilon^2|p|^2)^{-1/2}$ is indeed compact under our assumptions on $V$, hence it does not contribute to the polluted spectrum. On the other hand for $\epsilon=0$ it is the only term yielding pollution as we have seen before.

Theorem \ref{thm:no_poll_dual_basis} is then a consequence of Theorem \ref{thm:P} and of the following
\begin{lemma}[Properties of $(D^0+V)_{|\cP_\epsilon\cC}$ and $(D^0+V)_{|(1-\cP_\epsilon)\cC}$]\label{lem:prop_DKB}
The operators $A_1$ and $A_2$ defined in \eqref{def_op_DKB1} and \eqref{def_op_DKB2} are self-adjoint on the domain
$$\cD:=\left\{\phi\in L^2(\R^3,\C^2)\ |\ V(\sigma\cdot p)(1+\epsilon^2|p|^2)^{-1/2}\phi\in L^2(\R^3,\C^2)\right\}.$$
They are both essentially self-adjoint on $C^\ii_0(\R^3,\C^2)$. Moreover, we have
\begin{multline*}
\Conv\;{\rm Ess}\left(\frac2\epsilon-1+V\right)\subseteq\Conv\; \sigma_{\rm ess}\left(A_1\right)\subseteq\\
\subseteq\left[\min\left\{1\, ,\,\frac2\epsilon-1+\inf V\right\}\; ,\; \max\left\{1\, ,\,\frac2\epsilon-1+\sup V\right\}\right]
\end{multline*}
and
\begin{multline*}
\Conv\;{\rm Ess}\left(-\frac2\epsilon+1+V\right)\subseteq\Conv\; \sigma_{\rm ess}\left(A_2\right)\subseteq\\
\subseteq\left[\min\left\{-1\, ,\,-\frac2\epsilon+1+\inf V\right\}\; ,\; \max\left\{-1\, ,\,-\frac2\epsilon+1+\sup V\right\}\right]. 
\end{multline*}
\end{lemma}

\begin{proof}
The operator $K=(1+\epsilon^2|p|^2)^{-1/2}V(1+\epsilon^2|p|^2)^{-1/2}$ being compact, it suffices to prove the statement for $\cL_\epsilon(-2/\epsilon+2+V)\cL_\epsilon$, where we have introduced the notation $\cL_\epsilon:=\epsilon\sigma\cdot p(1+\epsilon^2|p|^2)^{-1/2}$. The argument is exactly similar for $\cL_\epsilon(2/\epsilon-2+V)\cL_\epsilon$. We denote $W:=-2/\epsilon+2+V$ and we introduce $A=\cL_\epsilon W\cL_\epsilon$ which is a symmetric operator defined on $\cD$. We also note that $\cD$ is dense in $L^2$.

Let $f\in D(A^*)$, i.e. such that $|\pscal{f,\cL_\epsilon W\cL_\epsilon\phi}|\leq C\norm{\phi}$, $\forall\phi\in \cD$. We introduce $\chi:=\1_{\cup_{k=1}^MB(R_k,r)}$, a localizing function around the singularities of $V$, and we recall that $V$ is bounded away from the $R_k$'s. Hence we also have $|\pscal{f,\cL_\epsilon \chi W\cL_\epsilon\phi}|\leq C'\norm{\phi}$ for all $\phi\in \cD$. Then we notice that under our assumptions on $V$, we have $W\chi\in L^2$, hence $g:=\chi W\cL_\epsilon f\in L^1$ and $\widehat{g}\in L^\ii$. In Fourier space the property $|\int\widehat{g}\widehat{\cL_\epsilon\phi}|\leq C'\norm{\phi}$ for all $\phi$ in a dense subspace of $L^2$ means that $\epsilon\sigma\cdot p(1+\epsilon^2|p|^2)^{-1/2}\widehat{g}(p)\in L^2$, hence $\widehat{g}\in L^2(\R^3\setminus B(0,1))$. As by construction $\widehat{g}\in L^\ii$, we finally deduce that $\widehat{g}\in L^2$, hence $W\cL_\epsilon f\in L^2$. We have proven that $D(A^*)\subseteq \cD$, hence $A$ is self-adjoint on $\cD$. The essential self-adjointness is easily verified.

The next step is to identify the essential spectrum of $A$. We consider a smooth normalized function $\zeta\in C^\ii_0(\R^3,\R)$ and we introduce $\phi_1=(1+\sigma\cdot p/|p|)(\zeta,0)$.
We notice that 
$\phi_1\in H^s(\R^3,\C^2)$ for all $s>0$. Then we let $\phi_n(x):=n^{3/2}\phi_1(n(x-x_0))$ and note that $(\sigma\cdot p/|p|)\phi_n=\phi_n$. We take for $x_0\in\R^3$ some fixed Lebesgue point of $V$, i.e. such that
\begin{equation}
\lim_{r\to0}\frac1{|B(x_0,r)|}\int_{B(x_0,r)}|V(x)-V(x_0)|dx=0.
\label{Lebesgue} 
\end{equation}
First we notice that
$$\lim_{n\to\ii}\norm{\left(\frac{\epsilon\sigma\cdot p}{\sqrt{1+\epsilon^2|p|^2}}-1\right)\phi_n}_{H^1}=\lim_{n\to\ii}\norm{\left(\frac{\epsilon|p|}{\sqrt{1+\epsilon^2|p|^2}}-1\right)\phi_n}_{H^1}=0$$
as is seen by Fourier transform and Lebesgue's dominated convergence theorem. Therefore, 
\begin{equation}
 \lim_{n\to\ii}\norm{W\left(\frac{\epsilon\sigma\cdot p}{\sqrt{1+\epsilon^2|p|^2}}-1\right)\zeta_n}_{L^2}=0
\label{limit_pscal_zeta_n}
\end{equation}
since we have $W\in L^2+L^\ii$. On the other hand we have
$\lim_{n\to\ii}\norm{(W-W(x_0))\zeta_n}_{L^2}=0.$
Using this to estimate cross terms we obtain $\lim_{n\to\ii}\norm{(A-W(x_0))\zeta_n}_{L^2}=0$. This proves that ${\rm Ess}(W)\subseteq \sigma_{\rm ess}(A)$. 
Let us remark that $0\in\sigma_{\rm ess}(A)$ as seen by taking $\phi'_n(x)=n^{-3/2}\phi_1(x/n)$.

The last step is to show that $\sigma_{\rm ess}(A)\subseteq [\min\{0,\inf(W)\},\max\{0,\sup(W)\}]$. When $\sup(W)<\ii$, we estimate
$A\leq \sup(W)\cL_\epsilon^2.$
If $W\leq0$, then we just get $A\leq 0$, hence $\sigma(A)\subseteq(-\ii,0]$. If $0<\sup(W)<\ii$, we can estimate $\cL_\epsilon^2\leq 1$ and get $\sigma(A)\subset(-\ii,\sup(W)]$. Repeating the argument for the lower bound, this ends the proof of Lemma \ref{lem:prop_DKB}.
\end{proof}
\end{proof}

\subsubsection{Dirac operators in free basis}\label{sec:Dirac2}
In this section, we prove that a way to avoid pollution in the \emph{whole gap} is to take a basis associated with the spectral decomposition of the free Dirac operator, i.e. choosing as projector $P^0_+:=\chi_{(0,\ii)}(D^0)$. As we will see this choice does not rely on the size of $V$ like in the previous section. Its main disadvantage compared to the dual method making use of $\cP_\epsilon$, is that constructing a basis preserving the decomposition induced by $P^0_+$ requires a Fourier transform, which might increase the computational cost dramatically.
First we treat the case of a `smooth' enough potential.
\begin{theorem}[No pollution in free basis - nonsingular case]\label{thm:no_poll_free_basis} Assume that $V$ is a real function such that 
$$V\in L^p(\R^3)+\left(L^r(\R^3)\cap \dot{W}^{1,q}(\R^3)\right)+L^\ii_\epsilon(\R^3)$$
for some $6<p<\ii$, some $3<r\leq 6$ and some $2<q<\ii$.
Then one has
$${\Spu(D^0+V,P^0_+)}=\emptyset.$$
\end{theorem}
\begin{remark}
We have used the notation
$$L^r(\R^3)\cap \dot{W}^{1,q}(\R^3)=\{V\in L^r(\R^3)\ |\ \nabla V\in L^q(\R^3)\}.$$ 
\end{remark}
\begin{remark}
A physical situation for which the potential $V$ satisfies the assumptions of the theorem is $V=\rho\ast\frac1{|x|}$ with $\rho\in L^1(\R^3)\cap L^2(\R^3)$.
\end{remark}

\begin{proof}
Under the above assumptions on the potential $V$, it is easily seen that the operator $D^0+V$ is self-adjoint with domain $H^1(\R^3,\C^4)$, the same as $D^0$, and that $\sigma_{\rm ess}(D^0+V)=\sigma(D^0)=(-\ii,-1]\cup[1,\ii)$ (these claims are indeed a consequence of the calculation below). Hence $(-1,1)$ only contains eigenvalues of finite multiplicity of $D^0+V$ and we may find a $c\in(-1,1)\setminus\sigma(D^0+V)$. In the following we shall assume for simplicity that $c=0$. The argument is very similar if $0\in\sigma(D^0+V)$. We will denote $\Pi=\chi_{[0,\ii)}(D^0+V)$ and prove that $(P^0_+ - \Pi)|D^0+V|^{1/2}$ is compact. This will end the proof, by Theorem \ref{thm:cond_no_pollute}.

As $0\notin\sigma(D^0+V)$, we have that $|D^0+V|\geq \epsilon$ for some $\epsilon>0$. Also, we have
$$\epsilon|D^0|^2+C_1\leq (D^0+V)^2\leq \epsilon|D^0|^2+C_2$$
for $\epsilon\geq0$ small enough. Taking the square root of the above inequality, this proves that $|D^0|^{-1/2}|D^0+V|^{1/2}$ and its inverse are both bounded operators.

 Next we use the resolvent formula together with Cauchy's formula like in \cite{HaiLewSer-05a} to infer
\begin{align*}
(P^0_+ - \Pi)|D^0+V|^{1/2} &=-\frac1{2\pi}\int_{-\ii}^\ii\left(\frac{1}{D^0+V+i\eta}-\frac{1}{D^0+i\eta}\right)|D^0+V|^{1/2}d\eta\\
&=\frac1{2\pi}\int_{-\ii}^\ii\frac{1}{D^0+i\eta}V\frac{|D^0+V|^{1/2}}{D^0+V+i\eta}d\eta.
\end{align*}

Let us now write  $V=V_1^n+V_2^n+V_3^n$ with $V_1^n\in L^p(\R^3)$ for $6<p<\ii$, $V_2^n\in L^r(\R^3)$ and $\nabla V_2^n\in L^q(\R^3)$ for $3<r\leq 6$, $2<q<\ii$, and $\norm{V_3^n}_{L^\ii(\R^3)}\to0$ as $n\to\ii$. We write
\begin{equation*}
 (P^0_+ - \Pi)|D^0+V|^{1/2} =K(V_1^n)+K(V_2^n)+K(V_3^n)
\end{equation*}
with
$$K(W):=\frac1{2\pi}\int_{-\ii}^\ii\frac{1}{D^0+i\eta}W\frac{|D^0+V|^{1/2}}{D^0+V+i\eta}d\eta$$
and estimate each term in an appropriate trace norm. We denote by $\gS_p$ the usual Schatten class \cite{Simon-79,ReeSim4} of operators $A$ having a finite $p$-trace, $\norm{A}_{\gS_p}=\tr(|A|^p)^{1/p}<\ii$. Let us recall the Kato-Seiler-Simon inequality (see \cite{SeiSim-75} and Thm 4.1 in \cite{Simon-79}) 
\begin{equation}
\forall p\geq2,\qquad \norm{f(-i\nabla)g(x)}_{\gS_p}\leq (2\pi)^{-3/p}
\norm{f}_{L^p(\R^3)}\norm{g}_{L^p(\R^3)}.
\label{KSS}
\end{equation}

The term $K(V_1^n)$ is treated as follows:
$$\norm{K(V_1^n)}_{\gS_p}\leq \frac1{2\pi}\int_{-\ii}^\ii\frac{d\eta}{(\epsilon^2+\eta^2)^{1/4}}\norm{(D^0+i\eta)^{-1}V_1^n }_{\gS_p}$$
where we have used that
$$\left\|\frac{|D^0+V|^{1/2}}{D^0+V+i\eta}\right\|\leq\frac{1}{(\epsilon^2+\eta^2)^{1/4}}.$$
By \eqref{KSS} we have
\begin{align}
\norm{(D^0+i\eta)^{-1}V_1^n}_{\gS_p}&\leq (2\pi)^{-3/p}\norm{V_1^n}_{L^p(\R^3)}\left(\int_{\R^3}\frac{dk}{(1+|k|^2+\eta^2)^{p/2}}\right)^{1/p}\nonumber\\
&\leq \frac{C}{1+\eta^{1-\frac{3}{p}}} \norm{V_1^n}_{L^p(\R^3)}.\label{premiere_estim_KSS}
\end{align}
Since $6<p<\ii$, this finally proves that 
$\norm{K(V_1^n)}_{\gS_p}\leq C\norm{V_1^n}_{L^p(\R^3)},$
hence this term is a compact operator for any $n$.

The term involving $V_2^n$ is more complicated to handle. First we use the formula \cite{HaiLewSer-05a,HaiLewSer-08}
\begin{align*}
\int_{-\ii}^\ii\frac{1}{D^0+i\eta}V_2^n\frac{1}{D^0+i\eta}\, d\eta&=\int_{-\ii}^\ii\frac{1}{D^0+i\eta}\left[V_2^n,\frac{1}{D^0+i\eta}\right]\, d\eta\\
&=\int_{-\ii}^\ii\frac{1}{(D^0+i\eta)^2}[D^0,V_2^n]\frac{1}{D^0+i\eta}\, d\eta\\
&=-i\int_{-\ii}^\ii\frac{1}{(D^0+i\eta)^2}(\alp\cdot \nabla V_2^n)\frac{1}{D^0+i\eta}\, d\eta.
\end{align*}
Iterating the resolvent formula we arrive at
\begin{multline}
K(V_2^n)=-\frac{i}{2\pi}\left(\int_{-\ii}^\ii\frac{1}{(D^0+i\eta)^2}(\alp\cdot \nabla V_2^n)\frac{|D^0|^{1/2}}{D^0+i\eta}d\eta\right)|D^0|^{-1/2}|D^0+V|^{1/2}\\
-\frac1{2\pi}\int_{-\ii}^\ii\frac{1}{D^0+i\eta}V_2^n\frac{1}{D^0+i\eta}V\frac{|D^0+V|^{1/2}}{D^0+V+i\eta}d\eta.\label{iterating_resolvent}
\end{multline}
The first term can be estimated as before by (recall that $|D^0|^{-1/2}|D^0+V|^{1/2}$ is bounded)
$$\norm{\frac{i}{2\pi}\int_{-\ii}^\ii\frac{1}{(D^0+i\eta)^2}(\alp\cdot \nabla V_2^n)\frac{|D^0|^{1/2}}{D^0+i\eta}d\eta}_{\gS_q}\leq C\norm{\nabla V_2^n}_{L^q(\R^3)}\int_{-\ii}^\ii\frac{d\eta}{1+\eta^{1+3\frac{q-2}{2}}}$$
which is convergent since $q>2$ by assumption.

The next step is to expand the last term of \eqref{iterating_resolvent} using again the resolvent expansion:
\begin{align}
 &\int_{-\ii}^\ii\frac{1}{D^0+i\eta}V_2^n\frac{1}{D^0+i\eta}V\frac{|D^0+V|^{1/2}}{D^0+V+i\eta}d\eta\nonumber\\
&\qquad=  \sum_{j=1}^{k-1}(-1)^{j+1}\left(\int_{-\ii}^\ii\frac{1}{D^0+i\eta}V_2^n\left(\frac{1}{D^0+i\eta}V\right)^j\frac{1}{D^0+i\eta}|D^0|^{1/2}d\eta\right)|D^0|^{-1/2}|D^0+V|^{1/2}\nonumber\\
&\qquad\qquad\qquad-(-1)^k\int_{-\ii}^\ii\frac{1}{D^0+i\eta}V_2^n\left(\frac{1}{D^0+i\eta}V\right)^k\frac{1}{D^0+V+i\eta}|D^0+V|^{1/2}d\eta.\label{seconde_estimate_KSS}
\end{align}
By \eqref{premiere_estim_KSS} we see that the last term belongs to $\gS_{kr}$ when $k$ is chosen large enough such that $k(1-3/r)>1/2$ (which is possible since $r>3$).

We now have to prove that the other terms corresponding to $j=2...k-1$ in \eqref{seconde_estimate_KSS} are also compact. We will only consider the term $j=2$, the others being handled similarly. Writing $V=V_1^n+V_2^n+V_3^n$ the terms containing $V_1^n$ and $V_3^n$ are treated using previous ideas. For the term which only contains $V_2^n$, the idea is, as done previously in \cite{HaiLewSer-05a}, to insert $P^0_++P^0_-=1$ as follows
$$\int_{-\ii}^\ii\frac{P^0_++P^0_-}{D^0+i\eta}V_2^n\frac{P^0_++P^0_-}{D^0+i\eta}V_2^n\frac{P^0_++P^0_-}{D^0+i\eta}|D^0|^{1/2}d\eta.$$
The next step is to expand and note that, by the residuum formula, the ones which contains only $P^0_+$ or only $P^0_-$ vanish. Hence we only have to treat terms which contain two different $P^0_\pm$. We will consider for instance
\begin{multline}
\int_{-\ii}^\ii\frac{P^0_-}{D^0+i\eta}V_2^n\frac{P^0_+}{D^0+i\eta}V_2^n\frac{P^0_+}{D^0+i\eta}|D^0|^{1/2}d\eta\\=\int_{-\ii}^\ii\frac{P^0_-}{D^0+i\eta}[P^0_-,V_2^n]\frac{P^0_+}{D^0+i\eta}V_2^n\frac{P^0_+}{D^0+i\eta}|D^0|^{1/2}d\eta.
\label{terme_deuxieme_ordre}
\end{multline}
Now we have using again a Cauchy formula for $P^0_-$
$$|D^0|^{-1/2}[P^0_-,V_2^n]=-\frac{i}{2\pi}\int_{-\ii}^\ii\frac{|D^0|^{-1/2}}{D^0+i\eta}\sigma\cdot\nabla V_2^n\frac{1}{D^0+i\eta}d\eta.$$
The Kato-Seiler-Simon inequality \eqref{KSS} yields as before
$$\norm{|D^0|^{-1/2}[P^0_-,V_2^n]}_{\gS_q}\leq C\norm{\nabla V_2^n}\int_{-\ii}^\ii\frac{d\eta}{1+\eta^{1+3\frac{q-2}{2}}}$$
Inserting this in \eqref{terme_deuxieme_ordre} and using that $V_n^2\in L^r$, we see that the corresponding operator is compact.

Eventually we have by a trivial estimate
$$\norm{(P^0_+ - \Pi)|D^0+V|^{1/2}-K(V_1^n)-K(V_2^n)}=\norm{K(V_3^n)}\leq C\norm{V_3^n}_{L^\ii(\R^3)}\to_{n\to\ii}0.$$
As $(P^0_+ - \Pi)|D^0+V|^{1/2}$ is a limit in the operator norm of a sequence of compact operators, it must be compact.
\end{proof}

We treat separately the case of a Coulomb-type singularity, for which $(P^0_+ - \Pi)|D^0+V|^{1/2}$ is not compact, hence we cannot use Theorem \ref{thm:cond_no_pollute} directly. 

\begin{theorem}[No pollution in free basis - Coulomb case] Let $|\kappa|<\sqrt{3}/2$. Then 
$${\Spu\left(D^0+\frac{\kappa}{|x|},P^0_+\right)}\cap(-1,1)=\emptyset.$$
\end{theorem}
\begin{proof}
 The operators $P^0_+(D^0+\kappa|x|^{-1})P^0_+$ and $P^0_-(D^0+\kappa|x|^{-1})P^0_-$ are known to have a self-adjoint Friedrichs extension as soon as $|\kappa|<2/(\pi/2+2/\pi)$, see \cite{EvaPerSie-96}. Furthermore one has $\sigma_{\rm ess}(D^0+\kappa|x|^{-1})_{|P^0_+L^2}=[1,\ii)$ and $\sigma_{\rm ess}(D^0+\kappa|x|^{-1})_{|P^0_-L^2}=(-\ii,-1]$, see Theorem 2 in \cite{EvaPerSie-96}. As $\sqrt{3}/2<2/(\pi/2+2/\pi)$, the result immediately follows from Theorem \ref{thm:P} and Remark \ref{rmk:ess_self-adjoint}.
\end{proof}

\section{Balanced basis}\label{sec:balanced}
In Section \ref{sec:splitting} we have studied and characterized spectral pollution in the case of a spitting $\gH=P\gH\oplus(1-P)\gH$ of the main Hilbert space. In particular for the case of the Dirac operator $D^0+V$ we have seen that the simple decomposition into upper and lower spinors may yield to pollution as soon as $V\neq0$. In this section we study an abstract theory (inspired of methods used in Physics and Chemistry) in which one tries to avoid pollution by \emph{imposing a relation between the vectors of the basis in $P\gH$ and in $(1-P)\gH$}, modelled by one operator $L:P\gH\to(1-P)\gH$. We call such basis a \emph{balanced basis}.

\subsection{General results}
Consider an orthogonal projection $P:\gH\to\gH$. Let $L:D(L)\subset P\gH\to(1-P)\gH$ be a (possibly unbounded) operator which we call \emph{balanced operator}. We assume that
\begin{itemize}
\item $L$ is an injection: if $Lx=0$ for $x\in D(L)$, then $x=0$;
\item $D(L)\oplus LD(L)$ is a core for $A$.
\end{itemize}

\begin{definition}[Spurious eigenvalues in balanced basis]\label{def_spur_balanced}\it
We say that $\lambda\in\R$ is a \emph{$(P,L)$-spurious eigenvalue} of the operator $A$ if there exist a sequence of finite dimensional spaces $\{V_n^+\}_{n\geq1}\subset D(L)$ with $V_n^+\subset V_{n+1}^+$ for any $n$, such that 
\begin{enumerate}
\item $\overline{\cup_{n\geq1}(V_n^+\oplus LV_n^+)}^{D(A)}=D(A)$;
\item $\displaystyle\lim_{n\to\ii}{\rm dist}\left(\lambda,\sigma\left(A_{|(V_n^+\oplus LV_n^+)}\right)\right)=0$;
\item $\lambda\notin\sigma(A)$.
\end{enumerate}
We denote by $\Spu(A,P,L)$ the set of $(P,L)$-spurious eigenvalues of the operator $A$.
\end{definition}

\begin{remark}
Another possible definition would be to only ask that for all $n$, $V_n^-$ contains $LV_n^+$. This would actually also correspond to some methods used by chemists (like the so-called \emph{unrestricted kinetic balance} \cite{DyaFae-90}). The study of these methods is similar but simpler than the one given by Definition \ref{def_spur_balanced}.
\end{remark}

Contrarily to the previous section, we will not characterize completely $(P,L)$-spurious eigenvalues. We will only give some necessary or sufficient conditions which will be enough for the examples we are interested in and which we study in the next section. We will assume as in the previous section that $PAP$ (resp. $(1-P)A(1-P)$) is essentially self-adjoint on $D(L)$ (resp. on $LD(L)$) with closure denoted as $A_{|P\gH}$ (resp. $A_{|(1-P)\gH}$).

\subsubsection{Sufficient conditions}\label{sec:balanced_suff}
We start by exhibiting a very simple part of the polluted spectrum. For any fixed $0\neq x\in D(L)$, we consider the $2\times2$ matrix $M(x)$ of $A$ restricted to the 2-dimensional space $x\C\oplus Lx\,\C$, and we denote by $\mu_1(x)\leq\mu_2(x)$ its eigenvalues. Note that $\mu_i$ is homogeneous for $i=1,2$, $\mu_i(\lambda x)=\mu_i(x)$.
\begin{theorem}[Pollution in balanced basis - sufficient condition]\label{thm:pollution_L_sufficient} Let $A$, $P$, $L$ as before and define $m_i,M_i\in\R\cup\{\pm\ii\}$, $i=1,2$, as follows:
\begin{equation}
m_1:=\inf_{\substack{\{x_n^+\}\subset D(L)\setminus\{0\},\\ x_n^+\norm{x_n^+}^{-1}\wto0,\\ Lx_n^+\norm{Lx_n^+}^{-1}\wto0}}\liminf_{n\to\ii}\mu_1(x_n^+),\qquad
M_1:=\sup_{\substack{\{x_n^+\}\subset D(L)\setminus\{0\},\\ x_n^+\norm{x_n^+}^{-1}\wto0,\\ Lx_n^+\norm{Lx_n^+}^{-1}\wto0}}\limsup_{n\to\ii}\mu_1(x_n^+),
\label{def_m_1} 
\end{equation}
\begin{equation}
m_2:=\inf_{\substack{\{x_n^+\}\subset D(L)\setminus\{0\},\\ x_n^+\norm{x_n^+}^{-1}\wto0,\\ Lx_n^+\norm{Lx_n^+}^{-1}\wto0}}\liminf_{n\to\ii}\mu_2(x_n^+),\qquad
M_2:=\sup_{\substack{\{x_n^+\}\subset D(L)\setminus\{0\},\\ x_n^+\norm{x_n^+}^{-1}\wto0,\\ Lx_n^+\norm{Lx_n^+}^{-1}\wto0}}\limsup_{n\to\ii}\mu_2(x_n^+).
\label{def_m} 
\end{equation}
Then we have:
\begin{equation}
[m_1,M_1]\cup [m_2,M_2]\subseteq \overline{\Spu(A,P,L)}\cup\hat\sigma_{\rm ess}(A).
\end{equation}
\end{theorem}

We supplement the above result by the following

\begin{remark}
The two diagonal elements of the matrix $A(x_n^+)$ being $\pscal{Ax_n^+,x_n^+}\norm{x_n^+}^{-2}$ and $\pscal{ALx_n^+,Lx_n^+}\norm{Lx_n^+}^{-2}$, it is clear that we have
$$m_2\geq m_1\geq \max\left(\inf\hat{\sigma}_{\rm ess}(A_{|(1-P)\gH})\,,\,\inf\hat{\sigma}_{\rm ess}(A_{|P\gH})\right),$$
$$M_1\leq M_2\leq \min\left(\sup\hat{\sigma}_{\rm ess}(A_{|(1-P)\gH})\,,\,\sup\hat{\sigma}_{\rm ess}(A_{|P\gH})\right),$$
which is compatible with Theorem \ref{thm:P}, since we must of course have $\Spu(A,P,L)\subset \Spu(A,P)$.
\end{remark}

\begin{proof}
We will use the following
\begin{lemma}\label{lem:reciproque_Weyl_L} 
Assume that $A$, $P$ and $L$ are as above.
Let $\{V_n\}\subset D(L)$ be a sequence of $K$-dimensional spaces with orthonormal basis $(x_n^1,...,x_n^K)$. Let $(y_n^1,...,y_n^K)$ be an orthonormal basis of $LV_n\subset (1-P)\gH$. We assume that $x_n^k\wto0$ and $y_n^k\wto0$ weakly for every $k=1..K$, as $n\to\ii$.
If $\lambda\in\R$ is such that 
$\lim_{n\to\ii}{\rm dist}\left(\lambda\,,\,\sigma(A_{|V_n\oplus LV_n})\right)=0,$
then $\lambda\in{\rm Spu}(A,P,L)\cup\sigma(A)$.
\end{lemma}
The proof of Lemma \ref{lem:reciproque_Weyl_L} will be omitted, it is very similar to that of Lemma \ref{lem:reciproque_Weyl}. We notice that the two sets
\begin{multline}
K_i:=\Big\{\mu\in\R\cup\{\pm\ii\}\ :\ \exists\{x_n^+\}\subset D(L),\ x_n^+\norm{x_n^+}^{-1}\wto0,\\ Lx_n^+\norm{Lx_n^+}^{-1}\wto0,\ \mu_i(x_n^+)\to\mu\Big\}. 
\end{multline}
are closed convex sets, for $i=1,2$. Indeed, assume for instance that $\lambda_1,\lambda_2\in K_1$ and let be $\{x_n\}$ and $\{y_n\}$ such that $\mu_1(x_n)\to\lambda_1$ and $\mu_1(y_n)\to\lambda_2$. By the homogeneity of $\mu_1$ we may assume that $\norm{x_n}=\norm{y_n}=1$ for all $n$. Also, extracting a subsequence from $\{y_n\}$, we may always assume that $$\lim_{n\to\ii}\pscal{x_n,y_n}=\lim_{n\to\ii}\pscal{\frac{Lx_n}{\norm{Lx_n}},\frac{Ly_n}{\norm{Ly_n}}}=0.$$

Fix some $\lambda\in(\lambda_1,\lambda_2)$ and consider as usual $z_n(\theta)=\cos\theta\; x_n+\sin\theta\;y_n$. By continuity of the first eigenvalue of the $2\times2$ matrix of $A$ in the space spanned by $z_n(\theta)$ and $Lz_n(\theta)$, we know that there exists (for $n$ large enough) a $\theta_n\in(0,2\pi)$ such that $\mu_1(\theta_n)=\lambda$. Note that $\norm{z_n(\theta_n)}=1+o(1)$. Writing $Lz_n(\theta_n)\norm{Lz_n(\theta_n)}^{-1}=\alpha_n Lx_n\norm{Lx_n}^{-1}+\beta_n Ly_n\norm{Ly_n}^{-1}$ we see that both $\alpha_n$ and $\beta_n$ are bounded and satisfy $\alpha_n^2+\beta_n^2\to1$, hence $\|Lz_n(\theta_n)\|\to1$. It is then clear that $z_n(\theta_n)\|z_n(\theta_n)\|^{-1}\wto0$ and that $Lz_n(\theta_n)\|Lz_n(\theta_n)\|^{-1}\wto0$. Therefore $\lambda=\lim_{n\to\ii}\mu_2(z_n(\theta_n))\in K_1$. The argument is the same for $K_2$.
As Lemma \ref{lem:reciproque_Weyl_L} tells us that $K_1\cup K_2\subset \Spu(A,P,L)\cup\sigma(A)$, this ends the proof of Theorem \ref{thm:pollution_L_sufficient}.
\end{proof}

\subsubsection{Necessary conditions}\label{sec:balanced_nec}
Let us emphasize that, contrarily to $P$-spurious eigenvalues, for $(P,L)$-spurious eigenvalues the two spaces $P\gH$ and $(1-P)\gH$ do not play anymore a symmetric role due to the introduction of the operator $L$. 
For this reason we shall concentrate on pollution occurring in the \emph{upper part of the spectrum} and we will not give necessary conditions for the lower part\footnote{As we have mentionned before we always assume for simplicity that $\inf\hat{\sigma}_{\rm ess}(A_{|(1-P)\gH})\leq \inf\hat{\sigma}_{\rm ess}(A_{|P\gH})$, i.e. that $1-P$ is responsible from the pollution occuring in the lower part of the spectrum.}. Loosely speaking, obtaining an information on the lower part would need to study the operator $L^{-1}$. In the applications of the next section, we will simply compute the lower polluted spectrum explicitely using Theorem \ref{thm:pollution_L_sufficient}.
Let us introduce
\begin{equation}
d:=\sup{\sigma}(A_{(1-P)\gH}).
\label{def_max_interval_left}
\end{equation}
and assume that $d<\ii$.
In the sequel we will only study $(P,L)$-spurious eigenvalues in $(c,\ii)$. Note that due to Theorem \ref{thm:P}, it would be more natural to let instead $d:=\sup \hat{\sigma}_{\rm ess}(A_{(1-P)\gH})$ but this will actually not change anything for the examples we want to treat: in the Dirac case $D^0+V$ and for $P=\cP$, the orthogonal projector on the upper spinor defined in \eqref{def_P_upper_lower_spinor}, the spectrum of $(D^0+V)_{|(1-P)L^2(\R^3,\C^4)}=-1+V$ is only composed of essential spectrum. We do not know how to handle the case of an operator $A_{|(1-P)\gH}$ which has a nonempty discrete spectrum above its essential spectrum.
Our main result is the following
\begin{theorem}[Pollution in balanced basis - necessary conditions]\label{thm:pollution_L_necessary}
Let $A$, $P$, $L$ as before. We recall that the real number $d<\ii$ was defined in \eqref{def_max_interval_left}.

\smallskip

\noindent $(i)$ Let us define
\begin{equation}
m_2''=\inf_{x^+\in D(L)\setminus\{0\}}\mu_2(x^+)
\label{def_m222}
\end{equation}
and assume that $m_2''>d$. Then we have
$$\Spu(A,P,L)\cap(d,m_2'')=\emptyset.$$

\smallskip

\noindent $(ii)$ Let us define 
\begin{equation}
m_2'=\inf_{\substack{\{x_n^+\}\subset D(L)\setminus\{0\},\\ x_n^+\wto0,\ \norm{x_n^+}=1}}\liminf_{n\to\ii}\mu_2(x_n^+)
\label{def_m22}
\end{equation}
and assume that $m_2'>d$.
We also assume that the following additional \emph{continuity property} holds for some real number $b>d$:
\begin{equation}
 \left.\begin{array}{r}
\{x_n^+\}\subset D(L)\\
x_n^+\to 0\\
\displaystyle \limsup_{n\to\ii}\mu_2(x_n^+)<b
\end{array}\right\}\Longrightarrow \pscal{Ax_n^+,x_n^+}\to0.
\label{property_P} 
\end{equation}
Then we have
$${\Spu(A,P,L)}\cap\big(d,\min(m_2',b)\big)=\emptyset.$$
\end{theorem}

\begin{remark}
The property \eqref{property_P} is a kind of compactness property at 0 of the set $\{x^+\in D(L)\ |\ \mu_2(x^+)<b\}$ for the quadratic-form norm of the operator $A_{|P\gH}$.
\end{remark}

\begin{remark}
Note that \eqref{property_P} holds true for $b=+\ii>d$ when $A_{|P\gH}$ is a bounded operator.
\end{remark}

Theorem \ref{thm:pollution_L_necessary} has many similarities with the characterization of eigenvalues in a gap which was proved by Dolbeault, Esteban and Séré in \cite{DolEstSer-00} (where our number $d=\sup{\sigma}(A_{(1-P)\gH})$ was denoted by `$a$'). In particular the reader should compare the assumptions $d<m_2'$ and $d<m_2''$ with (iii) at the bottom of p. 209 in \cite{DolEstSer-00}. The proof indeed uses many ideas of \cite{DolEstSer-00}. Note that \cite{DolEstSer-00} was itself inspired by an important Physics paper of Talman \cite{Talman-86} who introduced a minimax principle for the Dirac equation in order to avoid spectral pollution.

\begin{proof}
Assume that $\lambda\in \Spu(A,P,L)\cap(d,\ii)$. We consider a Weyl sequence $\{x_n\}$ like in Lemma \ref{lem:Weyl}, i.e. such that
\begin{equation}
P_{V_n^+\oplus LV_n^+}(A-\lambda_n)x_n=0
\label{equation_x_n} 
\end{equation}
for some $x_n\wto0$ with $\norm{x_n}=1$ and some $\lambda_n\to\lambda$. We write $x_n=x_n^++x_n^-$ where $x_n^-=Ly_n^+$ for some $y_n^+\in V_n^+$. Now, like in \cite{DolEstSer-00} we consider the following functional defined on $LV_n^+$:
$$Q(x^-):=\pscal{A(x_n^++x^-),x_n^++x^-}-\lambda_n\norm{x_n^++x^-}^2.$$
Using the equation $P_{V_n^+\oplus LV_n^+}(A-\lambda_n)x_n=0$, we deduce that
$$\forall x^-\in LV_n^+,\qquad Q(x^-)=\pscal{(A-\lambda_n)(x^--x_n^-),x^--x_n^-}.$$
By definition of $d$ we obtain
\begin{equation}
\forall x^-\in LV_n^+,\qquad Q(x^-)\leq (d-\lambda_n)\norm{x^--x_n^-}^2.
\label{estim_Q} 
\end{equation}

Consider the $2\times2$ matrix $M(x_n^+)$ of $A$ restricted to $x_n^+\oplus Lx_n^+$ and recall that $\mu_2(x_n^+)$ is by definition its second eigenvalue, hence
$$\mu_2(x_n^+)=\sup_{\theta\in\R}\frac{\pscal{A(x_n^++\theta Lx_n^+),x_n^++\theta Lx_n^+}}{\norm{x_n^++\theta Lx_n^+}^2}$$
(the sup is not necessarily attained). 
There exists $\theta_n\in\R$ such that for $n$ large enough
\begin{equation}
\frac{\pscal{A(x_n^++\theta_n Lx_n^+),x_n^++\theta_n Lx_n^+}}{\norm{x_n^++\theta_n Lx_n^+}^2}\geq \mu_2(x_n^+)-1/n.
\label{estim_mu_2_below} 
\end{equation}
Inserting $x^-=\theta_nLx_n^+$ in \eqref{estim_Q} we obtain for $n$ large enough,
\begin{equation}
\left(\mu_2(x_n^+)-\lambda_n-1/n\right)\left(\norm{x_n^+}^2+\theta_n^2\norm{Lx_n^+}^2\right)+(\lambda_n-d)\norm{\theta_nLx_n^+-x_n^-}^2\leq0.
\label{estim_Q2} 
\end{equation}

Let us assume we are in case $(i)$ for which $m_2''>d$. Using the obvious estimate $\mu_2(x_n^+)\geq m_2''$ we see that if $\lambda\in(d,m_2'')$, then for $n$ large enough we must have $x_n^+=\theta_nLx_n^+=\theta_nLx_n^+-x_n^-=0$, thus $x_n=0$ which is a contradiction with $\norm{x_n}=1$. Hence $\Spu(A,P,L)\cap(d,m_2'')=\emptyset$.

Let us now treat case $(ii)$ for which we assume $m_2'>d$ and that \eqref{property_P} holds for some $b>d$. Let $\lambda\in\Spu(A,P,L)\cap(d,\min(b,m_2'))$. From \eqref{estim_Q2} we see that necessarily $\mu_2(x_n^+)\leq \lambda_n+1/n$ (except if 
$x_n=0$ which is a contradiction). Therefore we have  $\limsup_{n\to\ii}\mu_2(x_n^+)<b$. 

Assume first that $x_n^+\to0$ strongly. Using our assumption \eqref{property_P}, we deduce that $\lim_{n\to\ii}\pscal{Ax_n^+,x_n^+}=0$. Next we argue like in the 3rd step of the proof of Theorem \ref{thm:P}. First, taking the scalar product of \eqref{equation_x_n} with $x_n^+$, we deduce that $\lim_{n\to\ii}\pscal{Ax_n^-,x_n^+}=0$. Taking then the scalar product with $x_n^-$ we deduce that
$$\lim_{n\to\ii}\pscal{(A-\lambda_n)x_n^-,x_n^-}=0.$$
As $\pscal{(A-\lambda_n)x_n^-,x_n^-}\leq (d-\lambda+o(1))\norm{x_n^-}^2$ and $d-\lambda<0$ we deduce that $x_n^-\to0$ which is a contradiction with $\norm{x_n}=1$.

Hence we must have $x_n^+\nrightarrow0$, which implies that $x_n^+\norm{x_n^+}^{-1}\wto0$, up to a subsequence. Therefore we have $\liminf_{n\to\ii}\mu_2(x_n^+)=\liminf_{n\to\ii}\mu_2(x_n^+\norm{x_n^+}^{-1})\geq m_2'$ by definition of $m_2'$. Inserting this information in \eqref{estim_Q2}, we again arrive at a contradiction, similarly as before.
This ends the proof of Theorem \ref{thm:pollution_L_necessary}.
\end{proof}

\subsection{Application to Dirac operator}
In this section, we consider the Dirac operator $A=D^0+V$ for a potential satisfying the assumptions \eqref{assumption_V_1} and \eqref{assumption_V_2} of Theorem \ref{thm:upper_lower_spinors} and 
\begin{equation}
\sup(V)<2
\label{ass_sup_V} 
\end{equation}
We will indeed for simplicity concentrate ourselves on the case for which either $V$ is bounded, or $V$ is a purely attractive Coulomb potential, $V(x)=-\kappa/|x|$, $0<\kappa<\sqrt3/2$. The generalization to potentials having several singularities is rather straightforward.

Like in Section \ref{sec:Dirac1}, we start by choosing $P=\cP$, the projector on the upper spinors as defined in \eqref{def_P_upper_lower_spinor}.
As already noticed in Section \ref{sec:Dirac1} we then have $\cP A\cP=1+V$ and $(1-\cP)A(1-\cP)=-1+V$ on the appropriate domain. This shows that the number $d$ introduced in the previous sections is
$d=-1+\sup V<1$
by \eqref{ass_sup_V}.

We will now study different balanced operators $L$ which we have found in the Quantum Chemistry litterature. Note that we can always see $L$ as an operator defined on $2$-spinors $D(L)\subset L^2(\R^3,\C^2)$ with values in the same Hilbert space $L^2(\R^3,\C^2)$, which we will do in the rest of the paper.

We will describe the polluted spectrum $\Spu(D^0+V,\cP,L)$ using the results presented in the previous sections. We note that 
the number $\mu_2(\phi)$ is the largest solution to the following equation \cite{DolEstSer-00}
\begin{equation}
 \pscal{(1+V)\phi,\phi}+\frac{\big(\Re\pscal{L\phi,\sigma\cdot (-i\nabla)\phi}\big)^2}{\pscal{(\mu+1-V)L\phi,L\phi}}=\mu\norm{\phi}^2
\end{equation}
where the denominator of the second term does not vanish when $\mu_2(\phi)>d=\sup(V)-1$.
Note the term on the left is decreasing with respect to $\mu$, whereas the term on the right is increasing with respect to $\mu$. Hence we have $\mu_2(\phi)\geq1$ if and only if
\begin{equation}
\pscal{V\phi,\phi}+\frac{\big(\Re\pscal{L\phi,\sigma\cdot (-i\nabla)\phi}\big)^2}{\pscal{(2-V)L\phi,L\phi}}\geq0
\label{Hardy_any_L}
\end{equation}
where the denominator of the second term does not vanish due to \eqref{ass_sup_V}.
Note that \eqref{Hardy_any_L} takes the form of a Hardy-type inequality similar to those which were found in \cite{DolEstSer-00,DolEstLosVeg-04}. In the following we will have to study this kind of inequalities for sequences $\phi_n$ which converge weakly to 0. The Hardy inequalities of \cite{DolEstSer-00,DolEstLosVeg-04} will indeed be an important tool as we will see below.

Concerning the choice of the operator $L$, several possibilities exist, although the main method is without any doubt the so-called \emph{kinetic balance} which we will study in the next section. All the methods from Quantum Chemistry or Physics are based on the following formula for an eigenfunction $(\phi,\chi)$ with eigenvalue $mc^2+\lambda$ (we reintroduce the speed of light $c$ and the mass $m$ for convenience) and which we have already formally derived before in Section \ref{sec:Dirac1}:
\begin{equation}
\chi=\frac{c}{2mc^2+\lambda-V}\sigma\cdot(-i\nabla)\phi.
\label{relation_upper_lower2}
\end{equation}
This equation suggests that for an eigenvector to be represented correctly, the basis of the lower spinor should contain $c(2mc^2+\lambda-V)^{-1}\sigma\cdot(-i\nabla)$ applied to the elements of the basis for the upper spinor. However we cannot choose in principle $L=c(2mc^2+\lambda-V)^{-1}\sigma\cdot(-i\nabla)$ because $\lambda$ is simply unknown. For this reason, one often takes the first order approximation in the nonrelativistic limit which is nothing but
$$\LKB=\frac{1}{2mc}\sigma\cdot(-i\nabla).$$
The choice of this balanced operator is (by far) the most widespread method in Quantum Physics and Chemistry. It will be studied in details in Section \ref{sec:kinetic_balance}.

It seems a well-known fact in Quantum Chemistry and Physics \cite{DyaFae-90,Pestka-03} that the kinetic balance method consisting in choosing $L=\LKB$ is not well-behaved for pointwise nuclei. The reason is that the behaviour at zero of $c(2mc^2+\lambda-V)^{-1}\sigma\cdot(-i\nabla)$ is not properly captured by $\sigma\cdot(-i\nabla)$, if $V(x)=-\kappa|x|^{-1}$. Indeed we will prove that the kinetic balance method allows to avoid pollution in the upper part of the spectrum for `regular' potentials, but not for Coulomb potentials, which justifies the aforementioned intuition.

To better capture the behaviour at zero, we study another method in Section \ref{sec:atomic_balance} which we call \emph{atomic balance}\footnote{The relation \eqref{relation_upper_lower} is usually called \emph{exact atomic balance}.} and which consists in choosing
$$\LAB=\frac{c}{2mc^2-V}\sigma\cdot(-i\nabla).$$
Although this operator does not depend on $\lambda$, it will be shown to completely avoid pollution in the upper part of the spectrum, even for Coulomb potentials. It is very likely that any other reasonable choice with the same behaviour at zero would have the same effect but we have not studied this question more deeply.

%

In the following we again work in units for which $m=c=1$.

\subsubsection{Kinetic Balance}\label{sec:kinetic_balance}
The most common method is the so-called \emph{kinetic balance} \cite{DraGol-81,Grant-82,Kutzelnigg-84,StaHav-84}. It consists in choosing as balanced operator 
\begin{equation}
\boxed{\LKB=-i\sigma\cdot\nabla}
\label{def:L_kinetic_balance} 
\end{equation}
We can for instance define $\LKB$ on the domain $D(\LKB)=C_0^\ii(\R^3,\C^2)$, in which case $\LKB$ satisfies all the assumptions of Section \ref{sec:balanced}.
Our main result is the following

\begin{theorem}[Kinetic Balance]\label{thm:kinetic_balance}
 \textbf{(i) Bounded potential.} Assume that $V\in L^p(\R^3)$ for some $p>3$, that $\lim_{|x|\to\ii}V(x)=0$, and that
\begin{equation}
-1+\sup(V)<1+\inf(V).
\label{cond_V}
\end{equation}
Then we have
$$\overline{\Spu(D^0+V,\cP,\LKB)}=[-1,-1+\sup V].$$
 
\medskip

\noindent \textbf{(ii) Coulomb potential.} Assume that $0<\kappa<\sqrt{3}/2$. Then we have
\begin{equation}
\overline{\Spu\left(D^0-\frac\kappa{|x|},\cP,\LKB\right)}=[-1,1].
\label{pollution_KB_Coulomb} 
\end{equation}
\end{theorem}

\begin{remark}
The conclusion \eqref{pollution_KB_Coulomb} also holds if $V$ is such that $V\in L^p(\R^3)\cap L^\ii(\R^3\setminus B(x_0,r))$ for some $p>3$ and 
$$-\frac{\kappa}{|x-x_0|}\leq V(x)\leq -\frac{\kappa'}{|x-x_0|^a}\text{ on } B(x_0,r)$$
for some $0<a\leq 1$ and some $\kappa<\sqrt{3}/2$, as is obviously seen from the proof.
\end{remark}

We have proved that the widely used kinetic balance method allows to \emph{avoid pollution in the upper part of the gap for smooth potentials}, hence for instance for $V=-\rho\ast|x|^{-1}$ where $\rho\geq0$ is the distribution of charge for smeared nuclei. However, the kinetic balance method \emph{does not avoid spectral pollution in the case of pointwise nuclei} (Coulomb potential).

\begin{proof}
\paragraph{Case \textbf{\textit{(i)}}.}
We assume that $V\in L^p(\R^3)\cap L^\ii(\R^3)$ satisfies \eqref{cond_V}. Clearly we have $\sup_\phi\mu_1(\phi)\leq -1+\sup(V)=:d$ and $m_2''=\inf_\phi\mu_2(\phi)\geq 1+\inf(V)$. Hence we necessarily have $m_2'\geq m_2''>d$ as requested by Theorem \ref{thm:pollution_L_necessary}. Also since $V$ is bounded by assumption, $(D^0+V)_{|\cP L^2(\R^3,\C^4)}=1+V$ is bounded, hence \eqref{property_P} holds for $b=1$. We deduce that
$$\Spu(D^0+V,\cP,\LKB)\cap(c,1)\subset [m_2',1).$$

Now we claim that $m'_2\geq1$. Indeed, let us argue by contradiction and assume that there exists a sequence $\phi_n\in C_0^\ii(\R^3,\C^2)$ such that $\phi_n\wto0$ in $L^2$, $\norm{\phi_n}=1$ and $\mu_2(\phi_n)\to\lambda\in(c,1)$. The number $\mu_2(\phi_n)$ is characterized by the equality
\begin{equation}
 \int_{\R^3}V|\phi_n|^2+\frac{\displaystyle\left(\int_{\R^3}|\sigma\cdot\nabla\phi_n|^2\right)^2}{\displaystyle\int_{\R^3}(\mu_2(\phi_n)+1-V)|\sigma\cdot\nabla\phi_n|^2}=\mu_2(\phi_n)-1.
\end{equation}
Since $V$ is bounded and $\norm{\phi_n}=1$ we get
$$\left|\int_{\R^3}|\sigma\cdot\nabla\phi_n|^2\right|^2\leq (1-\lambda+o(1)+\norm{V}_{\ii})(1-\lambda+o(1)+\norm{V}_{\ii})\int_{\R^3}|\sigma\cdot\nabla\phi_n|^2$$
which proves that $\{\phi_n\}$ is bounded in $H^1(\R^3,\C^2)$. We deduce that $\phi_n\wto0$ in $L^p(\R^3,\C^2)$ weakly for all $2\leq p\leq 6$ and strongly in $L^p_{\rm loc}(\R^3,\C^2)$ for all $2\leq p< 6$. Under our assumption on $V$, this shows that $\lim_{n\to\ii}\int V|\phi_n|^2=0$. For $n$ large enough, we thus have
\begin{equation}
\frac{\displaystyle\left(\int_{\R^3}|\sigma\cdot\nabla\phi_n|^2\right)^2}{\displaystyle\int_{\R^3}(\mu_2(\phi_n)+1-V)|\sigma\cdot\nabla\phi_n|^2}\leq \frac{\lambda-1}2<0
\end{equation}
which is a contradiction since by assumption $\mu_2(\phi_n)=\lambda+o(1)>d\geq V-1$. Hence we have proved that $\Spu(D^0+V,\cP,\LKB)\cap(d,1)=\emptyset$.

Now we assume $\sup(V)>0$ (otherwise there is nothing else to prove since $d=-1$) and prove that $(-1,-1+\sup(V)]\subset\overline{\Spu(D^0+V,\cP,\LKB)}$.
Let $x_0$ be a Lebesgue point of $V$, with $V(x_0)>0$ (hence $V(x)\geq0$ on a neighborhood of $x_0$). Consider a smooth radial nonnegative function $\zeta$ which is equal to 1 on the annulus $\{2\leq |x|\leq 3\}$ and 0 outside the annulus $\{1\leq |x|\leq 4\}$. We define for some fixed $\delta>0$
$$\phi_n(x)=\left(n^{1/2}\zeta\left(n(x-x_0)\right)+\frac{\delta^{1/2}}{(4n)^{3/2}}\zeta\left(\frac{x-x_0}{4n}\right)\right)\left(\begin{array}{c}1\\ 0 \end{array}\right) $$
where we have chosen the scaling in such a way that the above two functions have a disjoint support. We note that 
$$\int|\phi_n|^2=\delta N+O(n^{-2}),\qquad \int|\sigma\cdot\nabla\phi_n|^2=D+O(\delta n^{-2})$$
where we have introduced $N:=\int\zeta^2$ and $D=\int|\nabla\zeta|^2$.
Similarly, we have, using \eqref{Lebesgue} and that $V\to0$ at infinity,
$$\pscal{(1+V)\phi_n,\phi_n}=\delta (N+o(1))+O(n^{-2}),$$
$$\pscal{(-1+V)\LKB\phi_n,\LKB\phi_n}=(-1+V(x_0))D+O(n^{-2}).$$
Hence the matrix of $D^0+V$ in the basis $\{(\phi_n,0),(0,\LKB\phi_n)\}$ converges as $n\to\ii$ towards the following $2\times2$ matrix:
$$\left(\begin{matrix}
1 & \left(\frac{D}{N\delta}\right)^{1/2}\\
\left(\frac{D}{N\delta}\right)^{1/2} & -1+V(x_0)
\end{matrix} \right).$$
Eventually we note that $\phi_n\norm{\phi_n}^{-1}\wto0$ and $\sigma\cdot\nabla\phi_n\norm{\sigma\cdot\nabla\phi_n}^{-1}\wto0$. Hence, varying $\delta$ and $x_0$, we see that $M_1=-1+\sup(V)$ and $m_1\leq-1$ where $m_1$ and $M_1$ were defined in \eqref{def_m_1}. This ends the proof of \textbf{\textit{(i)}}, by Theorem \ref{thm:pollution_L_sufficient}.

\paragraph{Case \textbf{\textit{(ii)}}.} We will use again Theorem \ref{thm:pollution_L_sufficient}. More precisely we will show that $m_2=-\ii<-1$ and $M_2\geq1$, where $m_2$ and $M_2$ have been defined in \eqref{def_m}. This time we define
\begin{equation}
\phi_n(x)=\left(n^{1/2}\zeta\left(nx)\right)+(\delta n)^{1/2}\zeta\left(\delta nx)\right)\right)\left(\begin{array}{c}1\\ 0 \end{array}\right)  
\label{def_phi_poll_Coulomb}
\end{equation}
where $\delta\geq4$ is a fixed constant (note the above two functions then have a disjoint support). Similarly as before, we compute
$$\int|\phi_n|^2=\frac{1+\delta^{-2}}{n^2}N,\qquad \int|\sigma\cdot\nabla\phi_n|^2=2D,$$
$$\pscal{(1+V)\phi_n,\phi_n}=\frac{1+\delta^{-2}}{n^2}N-\kappa\frac{1+\delta^{-1}}{n}C_1,$$
$$\pscal{(-1+V)\LKB\phi_n,\LKB\phi_n}=-2D-\kappa(1+\delta)nC_2,$$
where $N$ and $D$ are defined as above and
$$C_1=\int_{\R^3}\frac{|\zeta(x)|^2}{|x|}dx,\qquad C_2=\int_{\R^3}\frac{|\sigma\cdot\nabla\zeta(x)|^2}{|x|}dx.$$
Hence, the matrix of $D^0-\kappa|x|^{-1}$ in the associated basis reads
$$A_n(\delta):=\left(\begin{matrix}
1-\kappa n\frac{1+\delta^{-1}}{1+\delta^{-2}}\frac{C_1}{N} & n\left(\frac{2D}{(1+\delta^{-2})N}\right)^{1/2}\\
n\left(\frac{2D}{(1+\delta^{-2})N}\right)^{1/2} & -1-\kappa(1+\delta)n\frac{C_2}{2D}
\end{matrix}\right).$$
Let us now choose $\delta\geq4$ large enough such that
$\kappa^2(1+\delta^{-1})(1+\delta)C_1C_2-2D^2>0.$
Then
\begin{equation}
\det(A_n(\delta))=\frac{\kappa^2(1+\delta^{-1})(1+\delta)C_1C_2-2D^2}{(1+\delta^{-2})ND}n^2+O(n)
\label{asymp_determinant} 
\end{equation}
hence $\det(A_n(\delta))\to+\ii$ as $n\to\ii$. Note that the first eigenvalue $\mu_1(\phi_n)$ of $A_n(\delta)$ satisfies
$$\mu_1(\phi_n)\leq -1-\kappa(1+\delta)n\frac{C_2}{2D}$$
hence $\mu_1(\phi_n)\to-\ii$ as $n\to\ii$. Therefore we must have $\mu_2(\phi_n)<0$ for $n$ large enough. More precisely
$$\mu_1(\phi_n)\geq -1-\kappa(1+\delta)n\frac{C_2}{2D}-n\left(\frac{2D}{(1+\delta^{-2})N}\right)^{1/2}$$
therefore, multiplying by $\mu_2(\phi_n)$ and using \eqref{asymp_determinant} we deduce that
$$\mu_2(\phi_n)\leq -\frac{\kappa^2(1+\delta^{-1})(1+\delta)C_1C_2-2D^2}{\kappa(1+\delta)(1+\delta^{-2})C_2N/2+D\left(2(1+\delta^{-2})N\right)^{1/2}}n+O(1),$$
which eventually proves that $\mu_2(\phi_n)\to-\ii$.
As it is clear that $\phi_n\norm{\phi_n}^{-1}\wto0$ and $\sigma\cdot\nabla\phi_n\norm{\sigma\cdot\nabla\phi_n}^{-1}\wto0$, we have shown that $m_2=-\ii$.

The proof that $M_2\geq1$ is simpler, it suffices to use 
$$\phi_n(x)=n^{-3/2}\zeta\left(\frac{x}{n}\right)$$
whose associated matrix of $A$ reads
$$A_n:=\left(\begin{matrix}
1-\kappa\frac{C_1}{Nn} & \sqrt{\frac{D}{N}}\frac{1}{n}\\
\sqrt{\frac{D}{N}}\frac{1}{n} & -1-\kappa\frac{C_2}{Dn}
\end{matrix}\right).$$
Therefore the result follows from Theorem \ref{thm:pollution_L_sufficient}.
\end{proof}

\subsubsection{Atomic Balance}\label{sec:atomic_balance}
We have proved in the previous section that the kinetic balance method allows to avoid spectral pollution in the case of a smooth potential, but that it does not solve the pollution issue for a Coulomb potential. In this section we consider another method called \emph{atomic balance}. It consists in taking
\begin{equation}
\boxed{\LAB=\frac{1}{2-V}\sigma\cdot(-i\nabla)}
\label{def:L_atomic_balance} 
\end{equation}
where we recall that we have assumed $2>\sup(V)$. Provided that $V$ is smooth enough, we can define $\LAB$ on the domain $D(\LAB)=C_0^\ii(\R^3\setminus\{0\},\C^2)$, in which case $\LAB$ satisfies all the assumptions of Section \ref{sec:balanced}.
Our main result is the following

\begin{theorem}[Atomic Balance]\label{thm:atomic_balance}
Let $V$ be such that $\sup(V)<2$, $(2-V)^{-2}\nabla V\in L^\ii(\R^3)$ and
$$-\frac{\kappa}{|x|}\leq V(x)$$
for some $0\leq\kappa<\sqrt{3}/2$.
We also assume that the positive part $\max(V,0)$ is in $L^p(\R^3)$ for some $p>3$ and that $\lim_{|x|\to\ii}V(x)=0$. 
Then we have
$$\overline{\Spu(D^0+V,\cP,\LAB)}=[-1,-1+\sup V].$$
\end{theorem}

\begin{remark}\it
We define the operator $\LAB$ on $D(\LAB)=C^\ii_0(\R^3\setminus\{0\},\C^2)$. Note that under our assumptions on $V$ we have that $\LAB D(\LAB)$ is dense in $H^1(\R^3,\C^2)$ for the associated Sobolev norm, hence $\LAB$ satisfies the properties required in Section \ref{sec:balanced_nec}. The above conditions on $V$ are probably far from being optimal.
\end{remark}

\begin{remark}\it
The choice of `2' in the definition of $\LAB$ is somewhat arbitrary. It can be seen that our result still holds true if $\sup(V)<1$ and $\LAB$ is replaced by $(\theta-V)^{-1}\sigma\cdot p$ for some fixed $\theta\geq1$. The proof is the same when $\theta\geq2$ but it is slightly more technical when $1\leq\theta<2$. 
\end{remark}

As we will explain in the proof, a very important tool is the Hardy-type inequality:
\begin{equation}
\int_{\R^3}\frac{c^2|\sigma\cdot\nabla\phi(x)|^2}{c^2+\frac\nu{|x|}+\sqrt{c^4-\nu^2c^2}}dx+(c^2-\sqrt{c^4-\nu^2c^2})\int_{\R^3}|\phi(x)|^2dx\geq \nu\int_{\R^3}\frac{|\phi(x)|^2}{|x|}dx.
\label{Hardy_DES_bis}
\end{equation}
This inequality was obtained in \cite{DolEstSer-00} by using a min-max characterization of the first eigenvalue of $-ic\alp\cdot\nabla+c^2\beta-\nu/|x|$. Indeed \eqref{Hardy_DES_bis} is an equality when $\phi$ is equal to the upper spinor of the eigenfunction corresponding to the first eigenvalue in $(-1,1)$ of $-ic\alp\cdot\nabla+c^2\beta-\nu/|x|$. The inequality \eqref{Hardy_DES_bis} was then proved by a direct analytical method in \cite{DolEstLosVeg-04}.
Introducing $m=c(1+\sqrt{1-(\nu/c)^2})$ and $\kappa=\nu/c$ we can rewrite \eqref{Hardy_DES_bis} in the following form
\begin{equation}
\int_{\R^3}\frac{|\sigma\cdot\nabla\phi(x)|^2}{m+\frac\kappa{|x|}}dx+m\frac{1-\sqrt{1-\kappa^2}}{1+\sqrt{1-\kappa^2}}\int_{\R^3}|\phi(x)|^2dx\geq \kappa\int_{\R^3}\frac{|\phi(x)|^2}{|x|}dx.
\label{Hardy_DES}
\end{equation}
We now provide the proof of Theorem \ref{thm:atomic_balance}.

\begin{proof}
Let us first prove that when $\sup(V)>0$, then we have $(-1,-1+\sup V]\subset \overline{\Spu(D^0+V,\cP,\LAB)}$.
The proof is indeed the same as that of Theorem \ref{thm:kinetic_balance}: we define for some fixed $\delta>0$
$$\phi_n(x)=\left(n^{1/2}\zeta\left(n(x-x_0)\right)+\frac{\delta^{1/2}}{(4n)^{3/2}}\zeta\left(\frac{x-x_0}{4n}\right)\right)\left(\begin{array}{c}1\\ 0 \end{array}\right),$$
where $x_0$ is a Lebesgue point of $V$ such that $0<V(x_0)<2$. One can prove that the matrix of $D^0+V$ in  $\{(\phi_n,0),(0,\LAB\phi_n)\}$ converges as $n\to\ii$ towards the following $2\times2$ matrix:
$$\left(\begin{matrix}
1 & \left(\frac{D}{N\delta}\right)^{1/2}\\
\left(\frac{D}{N\delta}\right)^{1/2} & -1+V(x_0)
\end{matrix} \right).$$
Hence we have again, by Theorem \ref{thm:pollution_L_sufficient}, $(-1,-1+\sup V]\subset \overline{\Spu(D^0+V,\cP,\LAB)}$.

The second part consists in proving that there is no spectral pollution above $-1+\sup(V)$. 
As a first illustration of the usefulness of the Hardy-type inequality \eqref{Hardy_DES}, we start by proving the following
\begin{lemma}\label{lem:estim_mu_2}
We have 
\begin{equation}
m_2''=\inf_{\phi\in D(\LAB)}\mu_2(\phi)\geq 1-2\frac{1-\sqrt{1-\kappa^2}}{1+\sqrt{1-\kappa^2}}.
\label{estim_mu_2_kappa}
\end{equation}
\end{lemma}
\begin{remark}
We note that the right hand side of \eqref{estim_mu_2_kappa} is always $\geq1/3$ when $0\leq\kappa<\sqrt{3}/2$, and it converges to $1$ as $\kappa\to0$, as it should be. 
\end{remark}

\begin{proof}
The number $\mu_2(\phi)$ is the largest solution of the equation
\begin{equation}
\int_{\R^3}(1+V(x))|\phi(x)|^2+\frac{\left(\int_{\R^3}\frac{|\sigma\cdot\nabla\phi(x)|^2}{2-V(x)}dx\right)^2}{\int_{\R^3}\frac{(1+\mu-V(x))|\sigma\cdot\nabla\phi(x)|^2}{\left(2-V(x)\right)^2}dx}=\mu\int_{\R^{3}}|\phi(x)|^2dx.
\label{equation_mu_AB} 
\end{equation}
Clearly we must always have 
$$\mu_2(\phi)> \mu_c(\phi):=-1+\frac{\int_{\R^3}\frac{ V(x)}{(2-V(x))^2}|\sigma\cdot\nabla\phi(x)|^2dx}{\int_{\R^3}\frac{ |\sigma\cdot\nabla\phi(x)|^2}{(2-V(x))^2}dx}.$$
Let be $\mu_c(\phi)<\mu<1$. We estimate:
\begin{align}
& \int_{\R^3}(1+V(x)-\mu)|\phi(x)|^2dx+\frac{\left(\int_{\R^3}\frac{|\sigma\cdot\nabla\phi(x)|^2}{2-V(x)}dx\right)^2}{\int_{\R^3}\frac{(1+\mu-V(x))|\sigma\cdot\nabla\phi(x)|^2}{\left(2-V(x)\right)^2}dx}\nonumber\\
&\qquad\qquad\qquad\geq\int_{\R^3}(1+V(x)-\mu)|\phi(x)|^2dx+\int_{\R^3}\frac{|\sigma\cdot\nabla\phi(x)|^2}{2+\frac{\kappa}{|x|}}dx\nonumber\\
&\qquad\qquad\qquad\geq\left(1-2\frac{1-\sqrt{1-\kappa^2}}{1+\sqrt{1-\kappa^2}}-\mu\right)\int_{\R^3}|\phi(x)|^2dx\label{estim_below_Hardy}
\end{align}
where in the last line we have used \eqref{Hardy_DES} and the fact that $\kappa|x|^{-1}+V(x)\geq0$. From this we deduce that 
$$\mu_2(\phi)\geq \max\left(1-2\frac{1-\sqrt{1-\kappa^2}}{1+\sqrt{1-\kappa^2}}\;,\; \mu_c(\phi)\right).$$
This ends the proof of Lemma \ref{lem:estim_mu_2}
\end{proof}

The next step is to prove that property \eqref{property_P} is satisfied.
\begin{lemma}\label{lem:calcul_m_2_AB}
Property  \eqref{property_P} holds true for $b=\ii$: if $\{\phi_n\}\subset C^\ii_0(\R^3,\C^2)$ is such that $\phi_n\to0$ in $L^2$ and $\mu_2(\phi_n)\to\ell<\ii$, then $\int_{\R^3}V|\phi_n|^2\to0$.
\end{lemma}
\begin{proof}
Note that necessarily $\ell\geq1/3$ by Lemma \ref{lem:estim_mu_2}, hence $\ell$ must be finite. We use the estimate \eqref{estim_below_Hardy}, with $\mu=\mu_2(\phi_n)$ to get
\begin{equation}
0\geq \int_{\R^3}(1+V(x)-\mu_2(\phi_n))|\phi_n(x)|^2dx+\int_{\R^3}\frac{|\sigma\cdot\nabla\phi_n(x)|^2}{2+\frac{\kappa}{|x|}}dx.
\label{estim_mu_phi_main} 
\end{equation}
Now we write
\begin{align}
\int_{\R^3}\frac{|\sigma\cdot\nabla\phi(x)|^2}{2+\frac{\kappa}{|x|}}dx&=(1-\kappa^2)\int_{\R^3}\frac{|\sigma\cdot\nabla\phi(x)|^2}{2+\frac{\kappa}{|x|}}dx+\kappa \int_{\R^3}\frac{|\sigma\cdot\nabla\phi(x)|^2}{\frac2\kappa+\frac{1}{|x|}}dx\nonumber\\
&\geq(1-\kappa^2)\int_{\R^3}\frac{|\sigma\cdot\nabla\phi(x)|^2}{2+\frac{\kappa}{|x|}}dx+\kappa\int_{\R^3}\frac{|\phi(x)|^2}{|x|}dx-\int_{\R^3}|\phi(x)|^2dx\label{estim_below_Coulomb}
\end{align}
where we have used \eqref{Hardy_DES} with $m\leftrightarrow2/\kappa$ and $\kappa\leftrightarrow1$. We deduce that
\begin{equation}
\int_{\R^3}\left(\frac{\kappa}{|x|}+V(x)\right)|\phi_n(x)|^2dx+(1-\kappa^2)\int_{\R^3}\frac{|\sigma\cdot\nabla\phi_n(x)|^2}{2+\frac{\kappa}{|x|}}dx\leq \mu_2(\phi_n)\int_{\R^3}|\phi_n|^2.
\label{estim_above_mu_phi} 
\end{equation}
Using that $\mu_2(\phi_n)\to\ell$, that $V\geq-\kappa|x|^{-1}$ and $\phi_n\to0$ we deduce that 
$$\lim_{n\to\ii}\int_{\R^3}\frac{|\sigma\cdot\nabla\phi_n(x)|^2}{2+\frac{\kappa}{|x|}}dx=0.$$
Using again \eqref{estim_below_Coulomb} with $\phi=\phi_n$ we finally get the result.
\end{proof}

We will now prove the following 
\begin{lemma}
We have $m_2'\geq1$ where $m_2'$ was defined in \eqref{def_m22}.
\end{lemma}
\begin{proof}
Consider a sequence $\{\phi_n\}\subset C_0^\ii(\R^3,\C^2)$ such that $\norm{\phi_n}=1$ and $\phi_n\wto0$. We will argue by contradiction and suppose that, up to a subsequence, $\mu_2(\phi_n)\to\ell\in[1/3,1)$. Similarly as in the proof of Lemma \ref{lem:calcul_m_2_AB}, $\{\phi_n\}$ must satisfy \eqref{estim_above_mu_phi}, from which we infer that
$$\int_{\R^3}\frac{|\sigma\cdot\nabla\phi_n(x)|^2}{2+\frac{\kappa}{|x|}}dx\leq C,$$
hence $\{\phi_n\}$ is bounded in $H^1$. Therefore, up to a subsequence we may assume that $\phi_n\to0$ strongly in $L^p_{\rm loc}(\R^3)$ for $2\leq p<6$. 

Let us now fix a smooth partition of unity $\xi_0^2+\xi_1^2+\xi_2^2=1$ where each $\xi_i$ is $\geq0$, $\xi_0\equiv1$ on the ball $B(0,r)$ and $\xi_0\equiv0$ outside the ball $B(0,2r)$, $\xi_2\equiv1$ outside the ball $B(0,2R)$ and $\xi_2\equiv0$ in the ball $B(0,R)$. We fix $R$ large enough such that
$$\forall|x|\geq R,\qquad |V(x)|\leq \frac{1-\ell}{3}$$
and $r$ small enough such that
$$m-\epsilon\leq\frac{\epsilon}{2r}$$
where $\epsilon$ is a fixed constant chosen such that $1-\ell-\epsilon/3>(1-\ell)/3$ and $\kappa+\epsilon<\sqrt{3}/2$.

Next we use the (pointwise) IMS formula
$$|\nabla\phi(x)|^2=\sum_{i=0}^2|\nabla(\xi_i\phi)(x)|^2-|\phi(x)|^2\sum_{i=0}^2|\nabla\xi_i(x)|^2$$
and \eqref{estim_mu_phi_main} to infer, denoting $\phi_n^i:=\phi_n\xi_i$ and $\eta=\sum_{i=0}^2|\nabla\xi_i(x)|^2$,
\begin{multline}
\sum_{i=0}^2\left(\int_{\R^3}(1+V(x)-\mu_2(\phi_n))|\phi^i_n(x)|^2dx+\int_{\R^3}\frac{|\sigma\cdot\nabla\phi^i_n(x)|^2}{2+\frac{\kappa}{|x|}}dx\right)\\
\leq \int_{\R^3}\frac{\eta(x)|\phi_n(x)|^2}{2+\frac{\kappa}{|x|}}dx.\label{estim_all}
\end{multline}
Next we note that for $n$ large enough, by our definition of $R$,
\begin{equation}
\int_{\R^3}(1+V(x)-\mu_2(\phi_n))|\phi^2_n(x)|^2dx\geq \frac{1-\ell}{3}\norm{\phi_n^2}^2.
\label{estim_infini} 
\end{equation}
Similarly we have by definition of $r$ and $\epsilon$ (using that $\phi^0_n$ has its support in the ball $B(0,2r)$)
$$\int_{\R^3}\frac{|\sigma\cdot\nabla\phi^0_n(x)|^2}{2+\frac{\kappa}{|x|}}dx\geq \int_{\R^3}\frac{|\sigma\cdot\nabla\phi^0_n(x)|^2}{\epsilon+\frac{\kappa+\epsilon}{|x|}}dx\geq \kappa\int_{\R^3}\frac{|\phi_n^0(x)|^2}{|x|}-\frac{\epsilon}{3}\int_{\R^3}|\phi_n^0(x)|^2dx$$
where for the last inequality we have used \eqref{Hardy_DES} and $\kappa+\epsilon<\sqrt{3}/2$. Using again that $V\geq-\kappa|x|^{-1}$, we infer the lower bound, for $n$ large enough,
\begin{equation}
\int_{\R^3}(1+V(x)-\mu_2(\phi_n))|\phi^0_n(x)|^2dx+\int_{\R^3}\frac{|\sigma\cdot\nabla\phi^0_n(x)|^2}{2+\frac{\kappa}{|x|}}dx
\geq\frac{1-\ell}{3}\norm{\phi_n^0}^2.\label{estim_zero}
\end{equation}
Inserting \eqref{estim_infini} and \eqref{estim_zero} in \eqref{estim_all}, we obtain
$$\frac{1-\ell}{3}\left(\norm{\phi_n^2}^2+\norm{\phi_n^0}^2\right)\leq \int_{\R^3}\frac{\eta(x)|\phi_n(x)|^2}{2+\frac{\kappa}{|x|}}dx+\norm{V\1_{r\leq|x|\leq 2R}}_{L^\ii}\norm{\phi_n\1_{r\leq|x|\leq 2R}}_{L^2}^2.$$
Using the strong local convergence of $\phi_n$, we finally deduce that $\lim_{n\to\ii}\norm{\phi_n^2}=\lim_{n\to\ii}\norm{\phi_n^0}=0$ which is a contradiction with $\norm{\phi_n}=1$.
\end{proof}

The conclusion follows from Theorem \ref{thm:pollution_L_necessary} \textbf{\textit{(ii)}}. This ends the proof of Theorem \ref{thm:atomic_balance}.\hfill$\square$
\end{proof}

\subsubsection{Dual Kinetic Balance}\label{sec:dual_kinetic_balance}
Let us now study the method which was introduced in \cite{Shaetal-04}, based this time on the splitting of the Hilbert space induced by the projector $\cP_\epsilon$ defined in \eqref{def_P_epsilon}. We have seen in Theorem \ref{thm:no_poll_dual_basis} that pollution might occur when $\epsilon$ is not small enough.  We prove below that introducing a balance as proposed in \cite{Shaetal-04} does not in general decrease the polluted spectrum.

Let us introduce the following operator
$$J\left(\begin{array}{c}
\phi\\0
\end{array}\right)=\left(\begin{array}{c}
0\\ \phi
\end{array}\right)$$
defined on $\cP L^2(\R^3,\C^4)$ with values in $(1-\cP) L^2(\R^3,\C^4)$. Next we introduce the following balance operator \cite{Shaetal-04}
\begin{equation}
\boxed{L_{DKB}=U_\epsilon J U_\epsilon} 
\label{def_LDKB}
\end{equation}
which is an isometry defined on $\cP_\epsilon L^2(\R^3,\C^4)$ with values in $(1-\cP_\epsilon) L^2(\R^3,\C^4)$. A calculation shows that, like in \cite{Shaetal-04}, formulas (24) and (25),
$$\LDKB\left(\begin{array}{c}
\phi\\ \epsilon\sigma(-i\nabla)\phi
\end{array}\right)=\left(\begin{array}{c}
\epsilon\sigma(-i\nabla)\phi\\ -\phi
\end{array}\right).$$
As before we may define $\LDKB$ on $\cC=U_\epsilon C^\ii_0(\R^3,\C^4)$. 

\begin{theorem}[Dual Kinetic Balance]\label{thm:no_poll_dual_kinetic_basis}
Assume that the real function $V$ satisfies the same assumptions as in Theorem \ref{thm:upper_lower_spinors}. Assume also that $\cP_\epsilon$ and $\LDKB$ are defined as in \eqref{def_P_epsilon} and \eqref{def_LDKB} for some $0<\epsilon\leq1$. Then one has
\begin{multline*}
\overline{\Spu(D^0+V,\cP_\epsilon,\LDKB)}=\overline{\Spu(D^0+V,\cP_\epsilon)}\\
=\left[-1\; ,\; \min\left\{-\frac2\epsilon+1+\sup V\;,\;1\right\}\right]\cup\left[\max\left\{-1\;,\;\frac2\epsilon-1+\inf V\right\}\; ,\; 1\right].
\label{Spu_P_epsilon_bis}
\end{multline*}
\end{theorem}
\begin{proof}
We will use Theorem \ref{thm:pollution_L_sufficient}. Consider a radial function $\zeta\in C_0^\ii(\R^3,\R)$ and introduce the following functions: $\phi_1:=(\zeta,0)$ and $\phi_1':=(\sigma\cdot p)/|p|\phi_1\in\cap_{s>0}H^s(\R^3,\C^2)$. We define similarly as in the proof of Theorem \ref{thm:no_poll_dual_basis}, $\phi_n(x)=n^{3/2}\phi_1(n(x-x_0))$ and $\phi'_n(x)=n^{3/2}\phi'_1(n(x-x_0))$, where $x_0$ is a fixed Lebesgue point of $V$. We note that $\phi_n':=(\sigma\cdot p)/|p|\phi_n$. Also, using that $\widehat\zeta$ is radial, we get for any real function $f$:
\begin{equation}
\pscal{f(|p|)\phi_n,\phi_n'}=\pscal{f(n|p|)\phi_1,\phi_1'}=\int_{S^2}\omega_1d\omega\int_0^\ii|\widehat{\zeta}(|p|)|^2f(n|p|)|p|^2d|p|=0.
\label{prop_radial} 
\end{equation}

A simple calculation shows that the $2\times2$ matrix of $D^0+V$ in the basis $(U_\epsilon(\phi_n,0)\;,\;\LDKB U_\epsilon(\phi_n,0))$ reads
$$M_n=\left(\begin{matrix}
\pscal{A_{11}\phi_n,\phi_n}&  \pscal{A_{12}\phi_n,\phi_n}\\
\pscal{A_{21}\phi_n,\phi_n}&  \pscal{A_{22}\phi_n,\phi_n}\\
\end{matrix}\right)$$
where 
\begin{equation*}
A_{11}= 1+\frac{1}{\sqrt{1+\epsilon^2|p|^2}}V\frac{1}{\sqrt{1+\epsilon^2|p|^2}}
+\frac{\epsilon\sigma\cdot p}{\sqrt{1+\epsilon^2|p|^2}}\left(\frac2\epsilon-2+V\right)\frac{\epsilon\sigma\cdot p}{\sqrt{1+\epsilon^2|p|^2}},
\end{equation*}
\begin{equation*}
A_{22}= -1+\frac{1}{\sqrt{1+\epsilon^2|p|^2}}V\frac{1}{\sqrt{1+\epsilon^2|p|^2}}
+\frac{\epsilon\sigma\cdot p}{\sqrt{1+\epsilon^2|p|^2}}\left(-\frac2\epsilon+2+V\right)\frac{\epsilon\sigma\cdot p}{\sqrt{1+\epsilon^2|p|^2}},
\end{equation*}
\begin{equation*}
A_{12}=(A_{21})^*= \frac{2\epsilon-1+\epsilon^2|p|^2}{1+\epsilon^2|p|^2}(\sigma\cdot p)+\epsilon\frac{1}{\sqrt{1+\epsilon^2|p|^2}}[V,\sigma\cdot p]\frac{1}{\sqrt{1+\epsilon^2|p|^2}}.
\end{equation*}
We infer from \eqref{prop_radial} that
$$\pscal{\frac{2\epsilon-1+\epsilon^2|p|^2}{1+\epsilon^2|p|^2}(\sigma\cdot p)\phi_n,\phi_n}=0$$
for every $n$. Also we have 
$$\lim_{n\to\ii}\norm{\frac{\epsilon\sigma\cdot p}{\sqrt{1+\epsilon^2|p|^2}}\phi_n-\phi'_n}_{H^1}=0.$$
It is then easy to see that 
$$\lim_{n\to\ii}M_n= \left(\begin{matrix}
\frac{2}{\epsilon}-1+V(x_0) & 0\\
0 & -\frac{2}{\epsilon}+1+V(x_0)\\
\end{matrix}\right).$$
Note that $\LDKB(\phi_n,0)\wto0$ since $\LDKB$ is an isometry. The result follows from Theorem \ref{thm:pollution_L_sufficient}, by varying $x_0$. 
\end{proof}


\end{document}